\documentclass[12pt]{article}  

\usepackage{amssymb}

\usepackage[dvipsnames]{xcolor}					
\usepackage{color}								
\usepackage[toc,title]{appendix}

\usepackage{comment}							
\usepackage[numbers]{natbib}					
\usepackage{pdfpages}                           
\usepackage{xspace}								
\usepackage{epstopdf} 							
\usepackage{nomencl}							
\usepackage{etoolbox}							
\usepackage{longtable}							
\usepackage{authblk} 
\usepackage{rotating}							

\usepackage[ margin=2cm]{geometry}
\usepackage{bbm}
\usepackage{todonotes}

\renewcommand\nomgroup[1]{%
	\item[\bfseries
	\ifstrequal{#1}{M}{Matrices}{%
		\ifstrequal{#1}{V}{Vectors}{%
			\ifstrequal{#1}{Sets}{Sets}{}}}%
	]}

\usepackage{graphicx}							
\usepackage{multirow}							
\usepackage{paralist}							
\usepackage{verbatim}							
\usepackage{longtable}							
\usepackage{booktabs}							
\usepackage{subcaption}							
\usepackage{enumitem}

\makeatletter
\def\BState{\State\hskip-\ALG@thistlm}
\makeatother

\usepackage{tikz}								
\usetikzlibrary{graphs,graphs.standard,quotes}	
\usetikzlibrary{backgrounds}					
\usetikzlibrary{fit} 							

\usepackage[T1]{fontenc}
\usepackage{lmodern}
\usepackage[utf8]{inputenc}						
\usepackage{soul}								
\usepackage[normalem]{ulem}						
\usepackage{mathrsfs}							
\usepackage{dsfont}

\usepackage[ruled,vlined]{algorithm2e} 
\usepackage{setspace}
\DontPrintSemicolon
\SetKw{KwAnd}{and}
\SetKw{KwOr}{or}
\SetKw{KwTrue}{true}
\SetKw{KwFalse}{false}

\usepackage{forloop}							
\usepackage{ifthen}								

\colorlet{darkred}{red!80!black}
\colorlet{darkgreen}{green!60!black}
\colorlet{darkblue}{blue!80!black}
\colorlet{darkorange}{orange!70!black}
\definecolor{Green}{cmyk}{1, 0.2, 0.4, 0.1}
\definecolor{Orange}{cmyk}{0, 0.61, 0.87, 0}
\definecolor{Purple}{rgb}{0.75, 0.0, 1.0}
\definecolor{purple}{rgb}{0.5,0,0.5}
\definecolor{dgreen}{rgb}{0.1,0.9,0.5}
\definecolor{gabysgreen}{cmyk}{0.80, 0.1, 0.90, 0}
\newcommand{\red}[1]{{\color{black}#1}}
\newcommand{\blue}[1]{{\color{black}#1}}

\usepackage{amssymb}							
\usepackage[fleqn]{amsmath}						
\usepackage{amsthm}								
\usepackage{bm}									
\usepackage{sansmath}
\usepackage{mathtools}							
\usepackage{nicefrac}							
\usepackage{siunitx}							
\usepackage{xfrac}								

\usepackage{hyperref}							
\hypersetup{
	colorlinks,
	linkcolor={blue!50!black},
	citecolor={green!50!black},
	urlcolor={blue!80!black}
}
\usepackage[capitalise,noabbrev]{cleveref}		

\newcommand{\irg}[2]{[#1\!:\!#2]}

\newcommand{\R}{\mathbb{R}}						
\newcommand{\N}{\mathbb{N}}						
\newcommand{\COP}{\mathcal{COP}}				
\newcommand{\CPP}{\mathcal{CPP}}				

\newcommand{\x}{\vc{x}}							
\renewcommand{\u}{\vc{u}}						
\newcommand{\lrbr}[1]{\left\lbrace #1 \right\rbrace} 

\newcommand\ga{\alpha}

\newcommand\gd{\delta}

\newcommand\gl{\lambda}
\newcommand\ggl{\bm{\lambda}}

\def\x{\mathbf x}
\def\a{\mathbf a}
\def\y{\mathbf y}
\def\z{\mathbf z}

\def\v{\mathbf v}
\def\w{\mathbf w}

\def\oo{\mathbf o}
\def\e{\mathbf e}
\def\u{\mathbf u}
\def\b{\mathbf b}
\def\c{\mathbf c}
\def\d{\mathbf d}

\def\q{{\mathbf q}}

\def\AA{{\mathcal A}}
\newcommand\BB{{\mathcal B}}
\newcommand\CC{{\mathcal C}}
\newcommand\DD{{\mathcal D}}

\newcommand\FF{{\mathcal F}}
\newcommand\GG{{\mathcal G}}

\newcommand\KK{{\mathcal K}}
\newcommand\LL{{\mathcal L}}
\newcommand\NN{{\mathcal N}}

\def\SS{{\mathcal S}}

\newcommand\VV{{\mathcal V}}

\newcommand\Ab{{\mathsf{A}}}
\newcommand\Bb{{\mathsf B}}
\newcommand\Cb{{\mathsf C}}

\newcommand\Eb{{\mathsf E}}

\newcommand\Mb{{\mathsf M}}
\newcommand\Nb{{\mathsf N}}

\newcommand\Ob{{\mathsf O}}
\newcommand\Pb{{\mathsf P}}
\newcommand\Qb{{\mathsf Q}}

\newcommand\Sb{{\mathsf S}}

\newcommand\Wb{{\mathsf W}}
\newcommand\Xb{{\mathsf X}}
\newcommand\Yb{{\mathsf Y}}
\newcommand\Zb{{\mathsf Z}}

\newcommand{\tr}[1]{\mathrm{Tr}(#1)}								
\newcommand{\T}{^\mathsf{T}}	
\DeclareMathOperator{\conv}{conv}									
\DeclareMathOperator{\diag}{diag}									
\def\beq#1{$$}														
\DeclareMathOperator{\epi}{epi}										
\DeclareMathOperator{\interior}{int} 								
\DeclareMathOperator{\relint}{relint} 					 			
\DeclareMathOperator{\cl}{cl} 									    
\DeclareMathOperator{\dom}{dom}										

\def\ol#1{{\overline #1}}

\def\eps{\varepsilon}

\def\ol#1{{\overline #1}}

\def\bea#1{\begin{array}{#1}}
	\def\ea{\end{array}}
\def\ignore#1{}

\theoremstyle{plain}
\newtheorem{thm}{Theorem} 
\newtheorem{exmp}{Example} 
\newtheorem{asm}{Assumption}

\newtheorem{prop}[thm]{Proposition}
\newtheorem*{rem}{Remark}
\theoremstyle{definition}

\theoremstyle{remark}

\usepackage{fancyhdr}

\begin{document}
	\renewcommand\Authands{ and }
	\title{Finding quadratic underestimators for optimal\\ value functions of nonconvex all-quadratic problems\\ via copositive optimization}
	
	\author{Markus Gabl\footnote{corresponding author}  
		\footnote{Department of Mathematics, University of Vienna, Austria; email: {\tt markus@gabl.com}}
		\Authands  
		Immanuel M.~Bomze\footnote{VCOR and Data Science@Uni Vienna, University of Vienna, Austria; email:  {\tt  immanuel.bomze@univie.ac.at}} 
	}
	\maketitle
	\begin{abstract}
		Modeling parts of an optimization problem as an optimal value function that depends on a top-level decision variable is a regular occurrence in optimization and an essential ingredient for methods such as Benders Decomposition. It often allows for the disentanglement of computational complexity and exploitation of special structures in the lower-level problem that define the optimal value functions. If this problem is convex, duality theory can be used to build piecewise affine models of the optimal value function over which the top-level problem can be optimized efficiently. In this text, we are interested in the optimal value function of an all-quadratic problem (also called quadratically constrained quadratic problem, QCQP) which is not necessarily convex, so that duality theory can not be applied without introducing a generally unquantifiable relaxation error. This issue can be bypassed by employing copositive reformulations of the underlying QCQP. We investigate two ways to parametrize these by the top-level variable. The first one leads to a copositive characterization of an underestimator that is sandwiched between the convex envelope of the optimal value function and that envelope's lower-semicontinuous hull. The dual of that characterization allows us to derive affine underestimators. The second parametrization yields an alternative characterization of the optimal value function itself, which other than the original version has an exact dual counterpart. From the latter, we can derive convex and nonconvex quadratic underestimators of the optimal value function. In fact, we can show that any quadratic underestimator is associated with a dual feasible solution in a certain sense. 
	\end{abstract}
	\newpage


	\section{Introduction}
	
	
	Many areas of optimization are dealing with problems of the following form
	\begin{align}\label{eqn:GeneralizedTwoStage}
		\inf_{\x\in\R^{n_x}}
		\lrbr{
			f_0(\x)+\phi(\x)
			\colon 
			\x\in\FF_0	}\quad\mbox{ with }
		\FF_0\subseteq \R^{n_x}\, , 
	\end{align}
	where the function $\phi$ is given by an inner optimization problem:  
	\begin{align}\label{eqn:Recourse}
		\hspace{-0.5cm}
		\phi(\x)
		\coloneqq 
		\inf_{\y\in\R^{n_y}}
		\lrbr{
			\red{q}(\y,\x) 
			\colon
			\y\in\FF(\x)\,
		}\quad\mbox{with}\quad
		\FF(\x)\coloneqq\lrbr{\y\in\R^{n_y}\colon [\x\T,\y\T]\T\in\FF}
	\end{align} 
	where $\FF\subseteq\R^{n_x+n_y}$ \red{and $q$ is a quadratic function}. Modeling optimization problems in this particular manner often arise from the desire to disentangle sources of computational complexity. For example, it might be hard to optimize over $\x$ and $\y$ simultaneously, while~\eqref{eqn:Recourse} may be easy to solve for a given value for $\x$, or this formulation may exhibit other exploitable structure. 
 
 Classical examples for such a constellation are given by two-stage stochastic optimization, robust optimization, some instances of large-scale optimization, bilevel optimization, and many more, for example applications related to Benders decomposition~\cite{rahmaniani_benders_2017} which aim to build an approximate model of $\phi$. This is done via information gathered from evaluating $\phi$ systematically at different points, used to approximately solve~\eqref{eqn:GeneralizedTwoStage}, mostly in an iterative manner based on said model, which is improved in every iteration.  
	
	If the optimization model defining $\phi$ is convex, one can use duality theory to derive affine underestimators at any given point $\x$ from the optimal dual objective. However, building approximations for $\phi$ is usually much more challenging in the nonconvex case. The same is true for building nonlinear underestimators even in the convex regime. In this text, we will demonstrate that via copositive optimization theory, we can find affine and (convex and nonconvex) quadratic underestimators of $\phi$ if it is the optimal value function of a quadratically constrained, quadratic optimization problem (QCQP). 
	
	Hence, we focus on the case where the function $f$ is a, possibly \red{(separately in $\x$ and $\y$ or also jointly in $(\x,\y)$)} nonconvex, quadratic function in the variables $[\x\T,\y\T]\T \in\R^{n_x+n_y}$ and where the set-valued mapping $$\FF(\x)\coloneqq\lrbr{\y\in\R^{n_y}\colon [\x\T,\y\T]\T\in\FF}$$  as well as the set $\FF$ are not necessarily convex, possibly even disconnected, but of sufficient simplicity in a certain sense which we will elaborate later (see Assumptions \ref{asm:CharacterizeG} and \ref{asm:CompactFs}). Under these restrictions, we set out to achieve the following:
	\begin{itemize}
		\item We show how to find affine underestimators of such a $\phi$, which are derived from a copositive characterization of a certain convex underestimator $\widehat{\phi}$ of $\phi$. This new function $\widehat{\phi}$ is constructed from a suitable parameterization of the copositive reformulation of the underlying QCQP and it is shown to be also an overestimator of the closure of the convex envelope of $\phi$. Some important properties of $\widehat{\phi}$ are studied as well. 
		
		\item Switching to an alternative parameterization of said copositive reformulation, we propose a convexified characterization of $\phi$ from which we can derive convex and nonconvex quadratic underestimators. We will show that, in a certain sense, any quadratic underestimator of $\phi$ can be derived in this manner.  
		
		\item An attractive feature of our analysis is the fact that once the hurdle of finding a copositive reformulation of the underlying QCQP has been taken, the derivations involve merely dualizing an appropriate parametrization of the reformulations optimal value function in much the same way it is traditionally done in case when $\phi$ were convex, to begin with. Thus, copositive optimization theory nicely interfaces between convex and nonconvex optimization.
  
    \item \blue{Finally, we exemplify our theory in numerical experiments where we develop a Benders type method to tackle a certain nonconvex QCQP. Our experiments show that computational advantages are possible and merit further investigation in future research.}
    
	\end{itemize}

 \blue{
	\subsection{Relation to prior work}
	Our analysis will mostly rely on copositive optimization theory which has been used in the past to derive convex envelopes of quadratic functions over certain feasible sets, see \cite{anstreicher_convex_2012,locatelli_convex_2014,yildirim_alternative_2022} who characterized convex envelopes of quadratic functions over polyhedral sets. The present text can be seen as a generalization of their work, where we consider convex envelopes of optimal value functions over (at least in theory) arbitrary sets. 
	
	Optimal value functions of QCQPs have played a role in robust optimization as well as in semi-infinite optimization.  Robust optimization where the uncertain function quadratically depends on the uncertainty parameter has been studied in \cite{bomze_interplay_2021,fan_decision_2024,gokalp_robust_2017,xu_improved_2023}, where copositive optimization techniques have been used to derive conservative approximations and Benders-type solution schemes.  
	
	Similarly, techniques resembling copositive optimization have been used in order to tackle the implied optimal value functions in semi-infinite programming. For example \cite{cerulli_convergent_2022,oustry_convex_2023} used SDP-based relaxations of such optimal value functions in order to derive for linear  Benders-type cuts mentioned above.  
	
	These approaches deal with cases where the optimal value functions have feasible sets that are independent of $\x$, and they make some strong assumptions on the structure of the feasible set. In addition, they derive convex approximations and, in as much as they are concerned with cuts, they derive linear ones. Our analysis allows, at least in principle, for arbitrary feasible sets which may depend on $\x$. In addition, we also consider convex and non-convex quadratic cuts. Thus, our analysis unifies and extends prior work in this area. 

 }
	
	\section{Preliminaries, definitions, and notation}
	Throughout the text, matrices are denoted with sans-serif capitals (or normal-face capitals in case they are $1\times 1$), 
	and vectors are given as boldface lowercase letters. For instance, $\Ob$ and $\Eb$ are the matrices of all zeros and ones respectively while the vector of zeros/ones is $\oo$/$\e$ and the order will be clear from the context. $\mathrm{Diag}(\Sb)$ extracts the diagonal of a square $n\times n$ matrix $\Sb$ as a vector in $\R^n$, the $n$-dimensional Euclidean space with its positive orthant $\R^n_+$.
	
	Sets will mostly be referred to using calligraphic letters, and so will some particular mappings, but the context will clarify. For example: $\SS^n$ is the space of symmetric $n\times n$ matrices, $\SS^n_+$ will be the cone of positive semidefinite matrices. On that note, $\Ab\bullet\Bb = \tr{\Ab\T\Bb}$ is the Frobenius product, an inner product for both $\R^{n\times m}$ and $\SS^n$, $\|\Ab\|_F := \sqrt{\Ab\bullet\Ab}$ being the associated Frobenius norm. We abbreviate the index set $\irg{i}{j} := \lrbr{i,i+1,\dots,j-1,j}$, where $i<j$ are integer numbers and finally the standard simplex $\Delta_n\coloneqq \lrbr{\x\in\R^n_+\colon \e\T\x=1}$. For a set $\FF\subseteq\R^n$ we denote by $\cl\FF,\ \conv\FF,\ \red{\mathrm{cone}\FF}$ and $\relint\FF$ its closure, convex hull, \red{conic hull}, and relative interior as defined in \cite{rockafellar_convex_2015}.
	
	Similarly, for a convex function $f$, we denote its closure by $\cl f$, which is defined as the largest lower semi-continuous function majorized by $f$, if $f$ is proper (i.e. $f>-\infty$ everywhere and not constantly equal $+\infty$) and as the constant function $-\infty$, in case $f$ is improper. 
	We also denote the convex envelope by $\conv f$, which is the largest convex function majorized by $f$. Its closure, $\cl\conv f$, will be referred to as the closed convex envelope for convenience. 
	Further,
	$\dom f:= \lrbr{\x\in\R^n\colon f(\x)<\infty}$ is its effective domain. As an example, 
	we have
	\begin{equation}\label{domfi}     
		\dom\phi = \lrbr{\x\in\R^{n_x}\colon \mbox{there exists }\y\in\R^{n_y}\colon [\x\T,\y\T]\T\in\FF}\, . \end{equation}
	Finally, the adjoint of a linear function $f$ will be denoted via $f^*$. 
	
	\blue{Typically, the literature is focused on convex envelopes rather than {\em closed} convex envelopes. However, this additional regularization allows us to derive very general results by comparatively simple techniques. Since this object may be less familiar to some readers,  we add the following proposition, which states that the closed convex envelope relates to the original function in a similar manner as the convex envelope. }
	\begin{prop}\label{prop:clconv}
		Let $f\colon \R^n\rightarrow\R\cup\lrbr{-\infty,\infty}$ be an arbitrary extended real-valued function. Then $\cl\conv f$ is the largest closed convex function majorized by $f$.
	\end{prop}
	\begin{proof}
		Let us first consider the case where $\conv f$ is improper. Any closed convex underestimator of $f$ is also majorized by the largest convex underestimator $\conv f$, so  $\cl\conv f$  must then equal $-\infty$ everywhere. Otherwise, if $\conv f$ is proper, we have $\dom f\neq \emptyset$ (otherwise $f\equiv +\infty$ and so would be $\conv f$, rendering it improper in contradiction to our assumption). For any $\x\in \dom f$ and any underestimator $g$ of $f$ we therefore have $g(\x)\le f(\x)< +\infty$, and $\dom g \neq \emptyset$ as well. Now
		define 
		$$h(\x)\coloneqq \sup\lrbr{g(\x) \colon g \mbox{ is a closed convex underestimator of } f}\, .$$ 
		By definition, $h$ majorizes all closed convex underestimators of $f$ including $\cl\conv f$, hence $h\geq \cl \conv f$, which also implies that $h$ is proper as $\dom h\supseteq \dom f\neq \emptyset$. For the converse, we first observe by elementary arguments that $h$, as a pointwise supremum of closed and convex functions, is itself closed and convex. 
		By definition of the convex envelope we have $\conv f\geq h$, so that $\epi h\supseteq \epi \conv f$ implying that 
		$$
		\epi h = \epi \cl h = \cl\epi h\supseteq \cl\epi \conv f = \epi \cl\conv f\, ,
		$$
		so that $\cl\conv f\geq h$. In the above equalities, we twice used the fact that whenever $g$ is proper, for the epigraphs we have $\epi \cl g$ = $\cl \epi g$, see for instance \cite[Theorem 7.1]{rockafellar_convex_2015} and the discussion succeeding that theorem.
	\end{proof}
        \blue{
	\begin{exmp}
		In order to illustrate the above concepts, consider the following nonconvex function and its convex and closed convex envelopes:
		\begin{align*}
			\hspace{-1cm}
			f(x) 
			\coloneqq 
			\begin{cases}
				1-x^2 & \mbox{ if } x^2<1 \\
				1     & \mbox{ if } x^2=1 \\
				\infty& \mbox{otherwise}
			\end{cases},
			\quad 
			\conv f(x) 
			\coloneqq 
			\begin{cases}
				0 & \mbox{ if } x^2<1 \\
				1 & \mbox{ if } x^2=1 \\
				\infty& \mbox{otherwise}
			\end{cases},
			\quad 
			\cl\conv f(x) 
			\coloneqq 
			\begin{cases}
				0 & \mbox{ if } x^2\leq1 \\
				\infty& \mbox{otherwise}
			\end{cases}.
		\end{align*}
		We see that the closed convex envelope differs from the convex envelope merely at some boundary points of the effective domain given by $\dom f = \lrbr{x\in\R \colon x^2\leq 1}$. However, if the feasible set over which $f$ is optimized intersects $\dom f$ only at its boundary, this may contribute to the relaxation gap.  
	\end{exmp}
    }
 
	\subsection{Copositive optimization prerequisites}\label{sec:Copositive optimization prerequisites}
	To carry out our analysis we will make heavy use of copositive optimization techniques. The core idea of copositive optimization is that for a possibly nonconvex quadratic optimization problem over a QCQP-feasible set $\FF$ we have the equality 
	\begin{align}\label{eqn:chlifting}
		\hspace{-0.6cm}\inf_{\x\in \FF} 
		\lrbr{
			\x\T \Qb\x +2 \q\T\x
		} 
		= 
		\inf_{
			\substack{\x\in\R^n,\\ \Xb\in\SS^n}
		}
		\lrbr{
			\Qb\bullet\Xb+ 2\q\T\x
			\colon 
			\begin{bmatrix}
				1 & \x\T \\ \x & \Xb 
			\end{bmatrix} 
			\in \GG(\FF)
		}\, ,
	\end{align}
	where
	\begin{align*}
		\GG(\FF) := \mathrm{clconv} \left\lbrace
		\left [ \bea{c}
		1 \\ \x  
		\ea \right ]
		\left [ \bea{c}
		1 \\ \x  
		\ea \right ]\T
		: \x\in \FF \right\rbrace\, .
	\end{align*}
	A proof for~\eqref{eqn:chlifting} can be found for example in \cite[Theorem 1]{bomze_optimization_2023} but earlier treatments are found in \cite{anstreicher_convex_2012,burer_copositive_2009,eichfelder_set-semidefinite_2013}. 
	Hence, given a description of $\GG(\FF)$ we can solve a convex optimization problem in order to obtain a solution to the original QCQP. The most important result in this regard is given by the theorem below (see \cite{burer_copositive_2009,eichfelder_set-semidefinite_2013,kim_geometrical_2020}).
	
    \begin{thm}\label{thm:BurerCPP}
		Let $\KK\subseteq \R^n$ be a closed, convex cone and let $\LL \coloneqq \lrbr{\x\in\KK\colon \Ab\x=\b}$ be nonempty \red{and let} $\LL_{\infty} \coloneqq  \lrbr{\x\in\KK\colon \Ab\x=\oo}$. Further, let $\Qb_i\in\SS^n,\ \q_i \in \R^n, \ i\in\irg{1}{l}$ and define $\BB\coloneqq \lrbr{j \colon \Qb_i\e_j  \neq \oo \mbox{ or } \q_i\T\e_j \neq 0 \mbox{ for some }i\in\irg{1}{l}}$ , \red{where $\e_j$ is the $j$-th 
  column of the $n\times n$ identity matrix}. Assume that\\
  a)  $\x\T\Qb_i\x+2\q_i\T\x-\omega_i$ is nonnegative for any $ \x\in\LL \mbox{ and } i\in\irg{1}{l}$, and\\  
  b) $\d\in\LL_{\infty} \implies d_j = 0$ for all $j \in \BB$.\\[0.2em]
		Then, for $\FF =  \lrbr{\x\in\KK\colon  \Ab\x = \b,\ \x\T\Qb_i\x+2\q_i\T\x=\omega_i,\ i\in\irg{1}{l}}$ we have	
		\begin{align*}
			\GG(\FF) = 
			\lrbr{
				\begin{pmatrix}
					1 &\x\T\\ \x & \Xb
				\end{pmatrix} \in \CPP\left(\R_+\times\KK\right)
				\colon
				\begin{array}{l}
					\Ab\x = \b,\		
					\mathrm{Diag}\left(\Ab\Xb\Ab\T\right)= \b\circ\b,\\
					\Qb_i\bullet\Xb+2\q_i\T\x=\omega_i,\ i\in\irg{1}{l},
				\end{array} 				 				
			}. 
		\end{align*}
	\end{thm}	
	The characterization involves the so-called \textit{set-completely positive matrix cone }
	\begin{align*}
		\CPP(\KK)\coloneqq \conv\lrbr{\x\x\T\colon \x\in\KK}\subseteq \SS_+^n,
	\end{align*} 
	where the cone $\KK\subseteq\R^n$ is called the \textit{ground cone}.
	These matrix cones have been heavily studied in for example \cite{berman_completely_2003,bomze_solving_2002,dickinson_copositive_2013,dur_interior_2008,dur_copositive_2010} and many more (see \cite[Section 2.4]{bomze_uncertainty_2022} for a guide through literature). Some interesting properties are discussed and (not necessarily for the first time) proved in \cite[Proposition 11]{bomze_optimization_2023} such as the fact that the dual cones are given by  
	\begin{align*}
		\COP(\KK) \coloneqq\lrbr{\Mb\in\SS^n\colon \v\T\Mb\v\geq 0, \ \forall \v\in\KK}\supseteq \SS_+^n\, ,
	\end{align*}
	the so called \textit{set-copositive matrix cones} with respect to the ground cone $\KK$. Testing membership in any of these matrix cones is difficult except for special cases, e.g. $\CPP(\R^n_+) \subseteq \SS_+^n\cap\NN^n$ (and $\COP(\R^n_+)\supseteq\SS_+^n+\NN^n$) with equality for $n\leq4$ where $\NN^n\subseteq\SS^n$ is comprised of the nonnegative matrices (see \cite{maxfield_matrix_1962}).
	
	\subsection{A compact representation}
	
	To keep the text concise we will now introduce notation that will allow the core of the discussion to be presented in a quite neat way. First, we consider the matrix decomposition: 
	\begin{align*}
		\Pb \coloneqq 
		\begin{pmatrix}
			x_0 & \x\T &\y\T \\
			\x& \Xb & \Zb\T \\ 
			\y&\Zb & \Yb \\ 
		\end{pmatrix}\in\SS^{n} \ .
	\end{align*}
	where $n \coloneqq n_x+n_y+1$
        . Later in the text, we will study approximations and characterizations of $\phi$ that are optimization problems in $\Pb$. The following proposition will be instructive\red{, as it details the role of the constituents $\Xb,\ \Yb,\ \Zb$ and $x_0$}. Its straightforward proof is omitted. 
	\begin{prop}\label{prop:rankone}
		Suppose $\Pb$ is positive-semidefinite (psd) and of rank one. Then we have $\Pb = \z\z\T$ for some $\z\in\R^{n+1}$. Furthermore, if $\z\T = [1,\x\T,\y\T]$, then $\Xb = \x\x\T$, $\Zb =\y\x\T,\ \Yb = \y\y\T$ and $x_0=1$. 
	\end{prop}
	
	For these matrices, we define the linear functions 
    $q \colon  \SS^{n}\mapsto \R,$ and associated coefficient matrices 
    $\Qb\in\SS^{n} $ so that 
	\begin{align*}		
		q(\Pb) 
		\coloneqq&  \Ab\bullet\Xb+2\a\T\x+2\Bb\T \bullet \Zb +\Cb\bullet \Yb + 2\c\T\y\
		= \Qb\bullet\Pb\,.
	\end{align*}
	Note that we can always find values for the entries of $\Qb$ such that $q(\z\z\T) = f(\x,\y)$ whenever $\z\T = [1,\x,\y\T]\in\R^n$ and $f$ is the quadratic function in the definition of $\phi$.
	
	In order to ensure that the parameter $\x$ equals the constituent $\x$ of the matrix variable $\Pb$ we add the constraint $\BB_{1}(\Pb) = \x$ to the respective optimization problems where $\BB_1\colon \SS^{n} \mapsto \R^{n_x}$ is an appropriate projection. We make use of a similar functions which project $\Pb$ onto $(\x,\Xb)$ and onto $(\x,\Xb,\Zb)$, which we will call $\BB_2\colon \SS^{n} \mapsto \R^n \times \SS^{n_x}$ and  $\BB_3\colon \SS^{n} \mapsto\R^n \times \SS^{n_x}\times \R^{n_y\times n_x}$, respectively. 
	
	\subsection{Main assumptions}
	
	We now turn to discussing our two main assumptions.
	\begin{asm}\label{asm:CharacterizeG}
		The sets $\FF$ is such that $\GG(\FF)$ is an affine slice of matrix cone  $\CPP\subseteq \SS^n_+$:
		\begin{align*}
			\GG(\FF)= 
			\lrbr{
				\Pb \in\CPP
				\colon \AA\left(\Pb\right) = \b
			}\subseteq\SS^{n}\, ,
		\end{align*}
		where $\CPP\subseteq\SS_+^{n}$ is an 
		appropriate closed convex matrix cone (with  dual cone $\COP$),  an appropriate linear map $\AA \colon  \SS^{n} \mapsto \R^{m}$, and a vector $\b\in\R^{m}\ $. 
	\end{asm}
	\noindent
	\cref{thm:BurerCPP} in \cref{sec:Copositive optimization prerequisites} guarantees that this assumption is not vacuous and most of our analysis rests on it for our purpose of constructing underestimators. We have so far not made any assumptions on the structure of $\FF\,$, but in practice, restrictions are imposed by the fact that $\GG(\FF)$ is only known for specific instances of $\FF$. \blue{While \cref{thm:BurerCPP} already engulfs a quite large class of choices for $\FF$, additional results which can be found in \cite{bomze_optimization_2023,burer_gentle_2015} and references therein, illustrate that above assumption is of moderate strength.} 
	However, for some parts of our discussion, we chose to put an additional restriction into place. 
	\begin{asm}\label{asm:CompactFs}
		$\FF$ is a nonempty, compact set.
	\end{asm}
	\begin{prop}\label{prop:CompactFs}
		Under \cref{asm:CompactFs} the set $\GG(\FF)$ is also nonempty and compact, even if the "$\cl$"-operator were omitted from the definition. 
	\end{prop}
	\begin{proof}
		The statement follows from the fact that $\GG(\FF)$ is the convex hull of a compact set, the image of $\FF$ under a continuous transformation from $\R^n$ to $\SS^{n+1}$, and therefore nonempty but also compact by the Caratheodory/Krein theorem (see  \cite[Corollary 5.33, p.185]{aliprantis_infinite_2006}).
	\end{proof}
	
	This assumption is not strictly necessary for any part of our analysis, but significantly simplifies the flow of the arguments, especially (but not exclusively) as they pertain to strong conic duality. In particular, the assumption implies the following proposition 
	
	\begin{prop}
		Under \cref{asm:CompactFs} we have zero duality gap in the following sense:
		\begin{align}\label{eqn:Slater}
			\inf_{\Pb\in\SS^n}
			\lrbr{
				\Qb\bullet\Pb 
				\colon 
				\AA\left(\Pb\right) = \b,\ \Pb \in\CPP
			} 
			= 
			\sup_{\ggl\in\R^m}
			\lrbr{
				\b\T\ggl
				\colon 
				\Qb-\AA^*(\ggl)\in\COP
			},
		\end{align}
		because of the existence of a Slater point $\ggl_{SL}(\Qb)$ for the right hand side dual. This point persists as a dual Slater point even if the primal is made infeasible via additional constraints.
	\end{prop}
	\begin{proof}
		For the convenience of the reader, we provide an argument for this well-known statement \red{(see e.g. \cite[Proposition 2.8]{shapiro_duality_2001})}. 
		For any $\Qb$, the duality gap closes due to a dual Slater point attained at some $\ggl_{SL}(\Qb)\in\R^m$. Indeed, if it were the case that  $\lrbr{\Qb-\AA^*(\ggl)\colon \ggl\in\R^m}\cap\interior\COP = \emptyset$, by \cite[Theorem 11.2]{rockafellar_convex_2015} we had a $\ol{\Pb}$ such that 
		\begin{align*}
			\ol\Pb\bullet \Yb 
			\geq 0,
			\quad 
			\forall \ \Yb\in\interior\COP, \quad 
			\mbox{and}
			\quad 
			\ol\Pb\bullet \Yb  = 0,
			\quad 
			\forall\ \Yb=\Qb-\AA^*(\ggl)\colon \ggl\in\R^m.
		\end{align*}
		From the first inequality, we see that $\ol\Pb\in\CPP$, while the second implies that 
		\begin{align*}
			0=\ol\Pb\bullet(\Qb-\AA^*(\ggl)) = \ol\Pb\bullet\Qb-\ggl\T\AA(\ol\Pb), \ \forall \ggl\in\R^m,
		\end{align*}
		in particular, for $\ggl=\oo$ we get $\Qb\bullet\ol\Pb = 0$. \red{The condition then reads as $\ggl\T\AA(\bar{\Pb}) = 0, \ \forall \ggl\in\R^m$, and we can choose $\ggl = \AA(\bar{\Pb})$ to show that $\ggl\T\AA(\bar{\Pb}) = \|\AA(\bar{\Pb})\|^2_2 = 0$ implies $\AA(\bar{\Pb}) = \oo$}. By \cref{asm:CompactFs} the left hand side problem has a feasible solution $\Pb_0$ and for any $\mu\geq0$ we get that $\AA(\Pb_0+\mu\ol\Pb) = \AA(\Pb_0)+\AA(\mu\ol\Pb) = \b+\mu\oo = \b$ and $\Pb_0+\mu\ol\Pb\in\CPP$ so that $\ol\Pb$ is a direction of recession contradicting \cref{asm:CompactFs}.       
		The last statement follows from the fact that the dual multipliers of the additional constraints can be set to zero. 
	\end{proof}
    \blue{
    Later in the text, we will consider instances of the optimization problems in \eqref{eqn:Slater} with additional constraints, fixing parts of the primal variable $\Pb$ to inputs of optimal value functions. This may render the resulting optimization problem infeasible even if $\GG(\FF)\neq \emptyset$ as per \cref{asm:CompactFs}. However, courtesy of \cref{asm:CompactFs} we always avoid the doubly infeasible case so that the duality gap stays closed.
    }	
	\section{Underestimators based on copositivity}
	
	\subsection{Affine underestimators via closed convex envelopes}\label{sec:Underestimating via closed convex envelopes}
	
	It is known that convex reformulations based on $\GG(\FF)$ can be used to construct convex envelopes of the objective function over $\FF$ by parametrizing the reformulations by the vector variable (see \cite[Theorem 1]{anstreicher_convex_2012}, later generalized in \cite[Proposition 3]{yildirim_alternative_2022}).  This insight serves as inspiration for the main results of this section, as we seek to characterize closed convex envelopes for the optimal value function $\phi$. The following theorem presents some vital insights on $\cl\conv \phi$.
	
	\begin{thm}\label{thm:ClosedConvexEnvelope}
		We pose~\cref{asm:CharacterizeG} that $\GG(\FF) = \CPP \cap \mathcal{A}^{-1}(\b)$. Define
		\begin{align*}
			\begin{split}
				\widehat{\phi}(\x) 
				\coloneqq
				\inf_{\Pb\in\SS^{n}}
				\lrbr{					
					q(\Pb)
					\colon			\Pb\in\GG(\FF),\
					\BB_{1}(\Pb) = \x
				}\, .
			\end{split}
		\end{align*}
		Then the dual of the infimum problem is given by 
		\begin{align*}
			\theta(\x)
			\coloneqq
			\sup_{\w,\ggl}
			\lrbr{
				\b\T\ggl+\x\T\w 
				\colon 
				\Qb-\AA^*(\ggl)-\BB_{1}^*(\w) \in \COP
			}\, .
		\end{align*}
		Also, let $h\coloneqq \cl\conv\phi$ be the closed convex envelope of $\phi$, cf.~\cref{prop:clconv}.  Then the following statements hold. 
		\begin{enumerate}
			\item  
			$\widehat{\phi}$ is a convex underestimator for $\phi$ with $\conv\dom\phi \subseteq \dom \widehat{\phi} \subseteq \cl\conv\dom\phi$,
			and under \cref{asm:CompactFs} the reverse inclusions hold as well. 
			\item We have $\widehat{\phi}(\x) \geq h(\x)$ for all $\x\in\R^{n_x}$ and thus $\dom h \supseteq\dom \widehat\phi\supseteq \dom \phi \neq \emptyset$.
			\item If ${\phi}(\x_0) = -\infty$ for some $\x_0\in\R^{n_x}$ then $\widehat{\phi}(\x) = -\infty$ for all $\x\in\relint\dom\phi$, and $h\equiv -\infty$ everywhere. 
			\item If $\theta$ has a feasible solution (e.g.~under~\cref{asm:CompactFs}), then $h$ is proper; furthermore, 
			\item if $\theta(\x)$ exhibits no duality gap for all $\x\in\R^{n_x}$ (e.g.~under~\cref{asm:CompactFs}), then actually  $\widehat{\phi} = \theta= h$.
		\end{enumerate} 
	\end{thm}
	\begin{proof}
		
		We start by proving 1. For a given $\x\in\dom\phi$ let $\y$ be a feasible solution for the infimum problem.   
		By setting $\Xb = \x\x\T, \, \Zb = \y\x\T, \, \Yb = \y\y\T$, we get a feasible solution for $\widehat{\phi}(\x)$ with the same objective function value, so we have a lower bound and $\widehat\phi (\x)\le \phi(\x)$. Outside $\dom\phi$, the relation holds regardless.
		
		Now assume $\lrbr{\x_1,\x_2}\subseteq \dom\widehat{\phi}$  and $\gl\in[0,1]$. The  convex combination with weight $\gl$ of the associated improving sequences of $\widehat{\phi}(\x_1)$ and  $\widehat{\phi}(\x_2)$ give an improving sequence for $\widehat{\phi}\left(\gl\x_1+(1-\gl)\x_2\right)$ so that, since the objective is linear, we have \begin{align*}
			\widehat{\phi}\left(\gl\x_1+(1-\gl)\x_2\right)\leq \gl\widehat{\phi}(\x_1) + (1-\gl)\widehat{\phi}(\x_2)\, .   
		\end{align*}  
		Outside $\dom\widehat{\phi}$, the function $\widehat{\phi}=+\infty$, so that in total $\widehat{\phi}$ is convex, and by above arguments, a convex underestimator of $\phi$. It follows $\conv\dom\phi \subseteq \dom \widehat{\phi}$. For the second inclusion, consider an $\x\in \dom \widehat{\phi}$, so that there must be a feasible improving sequence $\Pb^i\in\GG(\FF)$, $i \in\N$, such that $q(\Pb^i)\rightarrow \widehat{\phi}(\x)<\infty$ as $i\to\infty$, and  moreover $\BB_{1}(\Pb^i) = \x$ for all $i$ by feasibility. Since $\Pb^i$ are themselves limit points of convex combinations of dyadic matrices \red{(i.e. matrices of the form $\z\z\T$)} in $\GG(\FF)$ by definition of that set, we may construct another sequence
		\begin{align*}
			\widehat{\Pb}^i = \sum_{\ell=1}^{K}\gl_i^\ell
			\begin{bmatrix}
				1 \\ \x_i^\ell \\ \y_i^\ell
			\end{bmatrix}
			\begin{bmatrix}
				1 \\ \x_i^\ell \\ \y_i^\ell
			\end{bmatrix}\T \quad\mbox{ with } 
			\begin{bmatrix}
				\x_i^\ell \\ \y_i^\ell
			\end{bmatrix}\in\FF, \ \ggl_i\in\Delta_K, \mbox{ for all } i \in\N.
		\end{align*}
		such that $\lim_{i\in\N}\|\widehat{\Pb}^i-\Pb^i\|_F = 0$ and in addition $\x_i^\ell\in\dom\phi$ for all $(\ell,i)\in\irg{1}{K} \times\N$. But we also get from the definition of the Frobenius norm that $\|\x-\sum_{\ell=1}^{K}\gl^\ell\x_i^\ell\|_2 = \|\BB_{1}\left(\widehat{\Pb}^i-\Pb^i\right)\|_2 \leq \|\widehat{\Pb}^i-\Pb^i\|_F$, and we see that indeed 
		$$ \x = \lim_{i\to \infty} \sum_{\ell=1}^{K}\gl^\ell\x_i^\ell \in\cl\conv\dom\phi\, .$$ 
		Under~\cref{asm:CompactFs}, the effective domain $\dom{\phi}$ is closed as it is the projection of a compact set by~\eqref{domfi}, so that the reverse inclusions follow.  
		
		Now, to prove 2., consider $h$, which by definition is a closed, convex underestimator of $\phi$. By \cite[Theorem 12.1.]{rockafellar_convex_2015} it is therefore given by $h(\x)=\sup_{(\v,v_0)\in\VV}\lrbr{\v\T\x+v_0}$ where $\VV\subseteq\R^{n_x+1}$ is an appropriate set of coefficients.  If $\VV = \emptyset$, we have $h\equiv -\infty$ and the statement follows trivially. So assume the contrary and pick an arbitrary $(\v,v_0)\in\VV$. By assumption, we have: 
		\begin{align*}
			\phi(\x)\geq h(\x)\geq \v\T\x+v_0, \quad   \forall \x\in\R^{n_x},
		\end{align*}
		which implies that 
		\begin{align*}
			0 \leq& 
			\inf_{\x\in\R^{n_x}}
			\lrbr{
				\phi(\x)-\v\T\x-v_0
			},   \\
			=&  
			\inf_{\x\in\R^{n_x},\y\in\R^{n_2}} 
			\lrbr{  
				\Ab\bullet\x\x\T+2\a\T\x
				+
				2\x\T\Bb\T \y +\y\T \Cb\y+2\c\T\y -\v\T\x-v_0 
				\colon 
				\left[\x\T,\y\T\right]\T\in\FF
			}\\
			=& \inf_{\x\in\R^{n_x}}
			\lrbr{
				\widehat{\phi}(\x) -\v\T\x-v_0
			},
		\end{align*}
		where the first equality is true by definition, the second holds by the definition of $\widehat{\phi}$ and \eqref{eqn:chlifting}. Considering the fact that $(\v,v_0)$ was picked arbitrarily from $\VV$,  we get that
		\begin{align*}
			\widehat{\phi}(\x)  \geq \v\T\x + v_0, \  \forall \x\in\R^{n_x}, \ \forall (\v,v_0)\in\VV,
		\end{align*}
		and consequently 
		\begin{align*}
			\widehat{\phi}(\x) \geq \sup_{(\v,v_0)\in\VV}\lrbr{\v\T\x + v_0}  = h(\x) , \  \forall \x\in\R^{n_x}.
		\end{align*}	
		
		To observe 3., let $\phi (\x_0) = -\infty$. From the first two observations, we get that $\widehat{\phi}(\x_0) = h(\x_0) = -\infty$ so that both are improper convex functions. By \cite[Theorem 7.2]{rockafellar_convex_2015} these functions are $-\infty$  everywhere except for some boundary points of their domain unless the domain is closed, in which case the function is $-\infty$ everywhere.

		Next, we prove 4. So let $(\w^*,\ggl^*)$ be a feasible solution for $\theta(\widehat{\x})$. Then, by \cref{asm:CharacterizeG} and weak duality and the prior observations, we have 
		\begin{align*}
			\phi(\x)
			&\geq  
			\sup_{\w,\ggl}
			\lrbr{
				\b\T\ggl+\x\T\w
				\colon 
				\Qb-\AA^*(\ggl)-\BB_{1}^*(\w) \in \COP
			}\\
			&\geq 
			\b\T\ggl^*+\x\T\w^*,
		\end{align*}
		where the final inequality follows since the feasible set of $\theta$ does not depend on $\x$.  Thus, the function is an affine underestimator for $\phi$, hence closed and convex. But this implies it is majorized by $h(\x)$ which is therefore proper, as point~2. implies
		$\dom h \neq \emptyset$.
		
		Finally, to prove 5.,  assume that $\theta(\x)$ has a feasible point. Then given that $\widehat{\phi}(\x) = \theta(\x)>-\infty$ for all $\x\in\R^{n_x}$, we have that $\widehat{\phi}$ is a pointwise supremum of affine functions and therefore closed. Hence, it is a closed convex function majorizing the closure of the convex envelope $h$, so that by definition the two must coincide.    
	\end{proof}
	
	This theorem implies that we can use feasible solutions of the dual $\theta$, to generate affine underestimators of $\phi$, which under mild conditions reproduce its closed convex envelope $\widehat{\phi}$. Such linear models can potentially be used to power cutting plane methods such as Benders decomposition. The following example illustrates some of the intricacies regarding the difference of the convex envelope of $\phi$, the closure of that envelope, and $\widehat{\phi}$. 
	
	\begin{exmp}
		Consider the following instance of the optimal value function 
		\begin{align*}
			\phi(x) \coloneqq& \inf_{y} \lrbr{2ax y+by^2+cy \colon fx+gy = d,\ x \geq 0, \ y\geq0}, \mbox{ so that }\\
			\widehat{\phi}(x) \coloneqq& \inf_{X,Y,Z,y} \lrbr{2aZ+bY+cy \colon 
				\begin{array}{l}
					fx+gy = d, \\
					f^2X+2fgZ+g^2Y = d^2,
				\end{array}
				\begin{bmatrix}
					1 &x & y \\
					x &X & Z \\
					y &Z & Y
				\end{bmatrix}\in\CPP(\R^3_+)
			}.
		\end{align*}
		Since the set $\FF$ of feasible $[x,y]$ is the positive orthant, $\GG(\FF)$ coincides with the constraints defining $\widehat{\phi}$ by \cref{thm:BurerCPP}.
		\begin{itemize}[leftmargin=0.2cm]
			\item First, consider the case where $f = g = d = c = 0$. This trivializes the linear constraint and eliminates the linear term in the objective. This setup makes for an easy demonstration that $\widehat{\phi}$ can be either the convex envelope, or the closed convex envelope, or neither of both. 
			\begin{itemize}[leftmargin=0.3cm]
				\item Let further $a,b\geq 0$. We find, that $\phi(x) = 0$ over its effective domain given by $\R_+$. The same is true for $\widehat{\phi}(x)$ since the conic constraint implies $x,Z,Y\geq 0$. Hence the latter function is indeed the convex envelope of the former which in this case happens to be closed. 
				
				\item  If $a= 0,\ b<0$ we would find a similar relation. Both functions are $-\infty$ across their identical effective domains (given by $\R_+$). Except, now $\widehat{\phi}(x)$ is no longer the closure of the convex envelope as that would be the function that is $-\infty$ everywhere, which is the only improper convex function that is also closed. However, incidentally, $\widehat{\phi}(x) = \phi(x)$ and the latter is its own convex envelope so that $\widehat{\phi}$ is as well.    
				
				\item Now, let $a<0,\ b = 0$, then 
				\begin{align*}
					\phi(x)	 = 
					\begin{cases}
						0 & \mbox{ if } x=0,\\
						-\infty &  \mbox{ if } x>0,\\
						\infty & \mbox{\red{otherwise}},
					\end{cases}
				\end{align*}
				but for the conic problem the ray 
				\begin{align*}
					\begin{bmatrix}
						1 & x & 0 \\
						x & x^2 & 0\\
						0 & 0 & 0 
					\end{bmatrix}+ \gl 
					\begin{bmatrix}
						0 & 0 & 0 \\
						0 & 1 & 1\\
						0 & 1 & 1 
					\end{bmatrix} \in \CPP\left(\R^{3}_+\right), \quad \gl \geq 0,
				\end{align*}
				whenever $x\geq 0$. Hence we get an improving ray, so that $\widehat{\phi}(x) = -\infty $ for $x\geq 0$ and $\infty$ otherwise. The closed convex envelope is $-\infty$ everywhere while $\phi$ is its own convex envelope so that $\widehat{\phi}$ is neither of those. 
			\end{itemize}

			\item Now let's consider  $f = g = d =  1, \ a = 0,\ c= 0$ so that, by feasibility considerations, the effective domain of both functions is given by the line segment $[0,1]$. For instance, we can see that $\widehat{\phi}(x)= \infty$ whenever $x\notin[0,1]$ since in that case the constraints $x+y=1$ cannot be fulfilled by any $x,y\geq 0$ as necessitated by the conic constraint. More so, the feasible set of the conic problem is compact, so that a dual Slater point must exist, guaranteeing strong duality, which renders $\widehat{\phi}$ to be the closed convex envelope of $\phi$ by \cref{thm:ClosedConvexEnvelope}. Also, $\phi$ becomes an optimization problem over a singleton and is thus evaluated directly as 
			\begin{align*}
				\phi(x)	 = 
				\begin{cases}
					bx^2-2bx+b & \mbox{ if } x\in[0,1],\\
					\infty & \mbox{\red{otherwise}},
				\end{cases}
			\end{align*}
                \red{where we plugged in $y=1-x$.}
			Note, that $\phi(1) = 0$, while $\phi(0) = b$. We investigate the convex and the nonconvex case separately. 
			\begin{itemize}
				[leftmargin=0.3cm]
				\item Let us set $b>0$ so that $\phi(x)$ is an overall convex function. We can evaluate $\widehat{\phi}(x)$ in the following way. Inside the effective domain $[0,1]$ we will now construct a relaxation of $\widehat{\phi}$ that turns out to be tight. To that end,
				consider the fact that a matrix is in $\CPP(\R^3_+)$ if it is doubly nonnegative, i.e.\ psd and nonnegative. This implies $Y\geq y^2  = (1-x)^2$.
				We will now consider a lower bound of $\widehat{\phi}(x)$ given by 
				\begin{align*}
					bx^2-2bx+b = \inf_{Y}\lrbr{bY \colon 
						Y\geq (1-x)^2
					}  \leq \widehat{\phi}(x)\leq \phi(x) = bx^2-2bx+b, \ \forall x\in[0,1],
				\end{align*}
				where the leftmost inequality is \cref{thm:ClosedConvexEnvelope}.1. This, together with the fact that outside the unit interval both functions are $\infty$, implies $\widehat{\phi}=\phi$ as it should since $\phi$ is closed and convex and $\widehat{\phi}$ is its closed convex envelopes by \cref{thm:ClosedConvexEnvelope}.5.

				\item So consider the case $b< 0$. First, observe the feasibility of  $Y~=~1~-~x,\ X~=~x,\ Z~=~0$, $y~=~1~-~x$. Indeed, we can use Sylvester's criterion to certify positive semidefiniteness whenever $x\in[0,1]$, and the linear constraints hold as well. The objective function value is $b(1-x)$. To certify optimality for $\widehat{\phi}$, observe that, due to the linear constraint on the lifted variables, we could increase $Y$ by $\eps_Y> 0$ only if we change $X$ to $X+\eps_X$ and $Z=0$ to $Z+\eps_Z$ with $\eps_Z\geq 0$, in which case the increase in $Y$ would be $\eps_Y = -(\eps_X+2\eps_Z)$, which implies $\eps_X< -2\eps_Z\leq0$. However, since $x$ and $y=1-x$ are fixed by constraints, there is no feasible perturbation of other entries in the matrix variable, whose determinant would then become $-(\eps_X+\eps_Z)^2<0$,
				so that the matrix could not be psd anymore. We see that 
				\begin{align*}
					\widehat{\phi}(x)	 =  
					\begin{cases}
						b(1-x) & \mbox{ if } x\in[0,1],\\
						\infty & \mbox{\red{otherwise}},
					\end{cases}
				\end{align*}
				which is precisely the largest convex function underestimating $\phi$ in that interval and it is indeed closed.   
			\end{itemize}     		
		\end{itemize}		
	\end{exmp}

	\subsection{Quadratic underestimators via exact convexification }\label{sec:Quadratic first stage cuts via exact second stage convexification}
	
	In this section, we will describe a convexified characterization of $\phi$ whose dual objective is quadratic in the first-stage variable $\x$. Hence, the so-produced underestimators will be (possibly nonconvex) quadratic functions that, as we will show in \cref{sec:Quadratic envelopes}, can exactly represent $\phi$ in the limit. The theoretical core of the discussion is the following theorem.  
	
	\begin{thm}\label{thm:Exact2ndStConvx}
		Under \cref{asm:CompactFs}, we have $\phi(\x)  = \inf_{\Pb}\lrbr{q(\Pb) \colon \Pb\in\GG_j(\x)},\ j = 1,2$ where  
		\begin{align*}
			\hspace{-1cm}
			\GG_1(\x) 
			\coloneqq& 
			\lrbr{\Pb\in\GG(\FF) 
				\colon \BB_{2}(\Pb) = (\x,\x\x\T) }, \\
			\GG_2(\x) 
			\coloneqq& 
			\lrbr{\Pb\in\GG(\FF) 
				\colon \BB_{3}(\Pb) = (\x,\x\x\T,\x\y\T) }.
		\end{align*}
		In fact, $\GG_1(\x)=\GG_2(\x)$ and both sets are unchanged in case $\Xb = \x\x\T$ is relaxed to $\diag(\Xb) = \x\circ\x$.  
	\end{thm}
	\begin{proof}
		We can always rewrite $\phi(\u)  = \inf_{\x,\y}\lrbr{f(\x,\y) \colon [\x\T,\y\T]\in\FF,\ \x = \u}$ so that by \eqref{eqn:chlifting} the first part of the theorem will follow if we can show that $\GG(\FF') = \lrbr{\Pb\in\GG(\FF) \colon \x = \u, \ \Xb = \u\u\T}$, where $\FF' = \lrbr{[\x\T,\y\T]\in\FF \colon \x = \u}$. We have $\GG(\FF')\subseteq\GG_1(\x)$ since any extreme point of the former set is of the form $\Pb = \z\z\T,\ \z = [1,\x\T,\y\T]\in\lrbr{1}\times\FF'$, so that $\Pb\in\GG(\FF')$ and the containment follows from the convexity of both sets. For the converse, from \cref{prop:CompactFs} and the definition of $\GG(\FF)$ we have $\Pb = \sum_{i=1}^{k}\gl_i\z_i\z_i\T,\ \z_i = [1,\x_i\T,\y_i\T]\in\lrbr{1}\times\FF,\ i \in\irg{1}{k}, \ \ggl\in\Delta_k$ so that $\Xb = \u\u\T, \ \x = \u$, which yields that
		\begin{align}\label{eqn:extremepoints}
			\hspace{-0.9cm}
			\begin{bmatrix}
				1 & \u\T \\ \u & \u\u\T
			\end{bmatrix}
			=             
			\sum_{i=1}^k
			\gl_i
			\begin{bmatrix}
				1 \\ \x_i
			\end{bmatrix}
			\begin{bmatrix}
				1 \\ \x_i
			\end{bmatrix}\T
			\implies
			\gl_i
			\begin{bmatrix}
				1 \\ \x_i
			\end{bmatrix}
			\begin{bmatrix}
				1 \\ \x_i
			\end{bmatrix}\T  = 
			\gd_i 
			\begin{bmatrix}
				1 \\ \u
			\end{bmatrix}
			\begin{bmatrix}
				1 \\ \u
			\end{bmatrix}\T,\, \gd_i \in\R_+, \, i\in\irg{1}{k}\, ,
		\end{align}
		by extremality of the rays generated by dyadic matrices in $\SS^{n_0}_+$ (see e.g. the proof of \cite[Proposition 11]{bomze_optimization_2023}). If $\gl_i = 0$ the respective term can be dropped from the sum so that we can assume $\gl_i>0$. Then \eqref{eqn:extremepoints} implies $\gd_i = \gl_i>0$ (from the north-west entry of the final equation) which further implies $\x_i = \u$ so that $\z_i\in\lrbr{1}\times \FF'$ as desired. Clearly, $\GG(\FF')\subseteq \GG_2(\x)$ by a similar argument as before (note that the constraint is linear in $\Pb$, since $\x$ is a constant) and $\GG_2(\x)\subseteq \GG_1(\x)$ by construction so that the two sets must coincide. Finally, said relaxation is tight since $\GG(\FF)\subseteq\SS_+^{n_x+n_y+1}$ so that by Schur complementation and the relaxed constraint we have $\Ob = \Xb -\x\x\T\in\SS^{n_x}$ since the zero matrix is the only psd with diagonal equal to zero. Hence, $\Xb = \x\x\T$ as desired. 
	\end{proof}
    \blue{
    \begin{rem}
        The equality $\GG_1(\x) = \GG_2(\x)$ shows that the conditions in the former set are enough to force $\Yb= \x\y\T$. This is a useful insight, which is for example used in~\cite{concavetent}. 
    \end{rem}
    }
 
	Under \cref{asm:CharacterizeG} the dual of the problem over $\GG_1(\x)$ is given by  
	\begin{align}\label{eqn:DualforQuadraticCuts}
		\sup_{\Wb,\w,\ggl}
		\lrbr{
			\b\T\ggl+\x\T\w+\x\x\T\bullet\Wb
			\colon 
			\Qb-\AA^*(\ggl)-\BB_{2}^*(\w,\Wb)\in \COP
		}, 
	\end{align}
	so that if we can close the duality gap, e.g.\ under \cref{asm:CompactFs}, we get a dual characterization of $\phi$. If the constraint $\Xb=\x\x\T$ is replaced with $\diag(\Xb) = \x\circ\x$ in the primal, then the dual variable $\Wb$ would be restricted to be a diagonal matrix. 
	
	We see that feasible solutions to \eqref{eqn:DualforQuadraticCuts} yield underestimators of $\phi$ that are quadratic in $\x$, where the convexity hinges on the value for $\Wb$.  It is therefore vital to understand whether it is possible to find a convex underestimator among the dual solutions. To this end, we will now study quadratic underestimators more closely.  
	
	\subsubsection{Quadratic envelopes}\label{sec:Quadratic envelopes}
	We will now investigate the properties of quadratic underestimators of $\phi$ and show that, in a certain sense, all such underestimators, convex or not, can be found among the solutions of \eqref{eqn:DualforQuadraticCuts}. A quadratic function $\widehat{q}$ underestimates $\phi$ if and only if 
	\begin{align}\label{eqn:QuadraticUnderEst}
		\phi(\x)\geq \widehat{q}(\x), \quad \forall \x\in\R^n\, .
	\end{align}
	Any quadratic function can be cast in the form $\w\T\x+\Wb\bullet\x\x\T+\nu$, where the choice of the coefficients denomination was purposeful as will be apparent shortly. Moreover, for $\Pb = \z\z\T$ with $\z\T = [1,\x\T,\y\T]$ we have 
	\begin{align*}
		\w\T\x+\Wb\bullet\x\x\T= \langle (\w,\Wb), \BB_{2}(\Pb) \rangle =  \langle \BB_{2}^*(\w,\Wb), \Pb \rangle = \BB_{2}^*(\w,\Wb)\bullet\Pb\, .
	\end{align*}
	With this in mind, we can now show that the relation between $\phi$ and their quadratic underestimators is almost mediated entirely by the dually feasible set of \eqref{eqn:DualforQuadraticCuts}. 
	\begin{thm}\label{thm:COPfeasibleQuadUEs}  
		Given \cref{asm:CharacterizeG}, the function 
		$\w\T\x+\Wb\bullet\x\x\T+\nu$  is a quadratic underestimator of $\phi$, if there are $\ggl^\ell\in\R^{m}, \ \ell \in\N$ such that $\ggl = \lim\limits_{\ell\rightarrow\infty}\left(\ggl^\ell\right)$ and 
		\begin{align*}
			&\b\T\ggl \geq \nu,\quad 
			\Qb-\AA^*(\ggl^\ell)-\BB_{2}^*(\w,\Wb) \in \COP, \ \ell\in\N\, .
		\end{align*}
		In addition, if the problem 
		\begin{align*}
			\inf_{\Pb\in\SS^{n}} 
			\lrbr{  
				q(\Pb)-\BB_{2}^*(\w,\Wb)\bullet\Pb \colon \AA(\Pb) =\b,\ \Pb\in\CPP
			},
		\end{align*}
		exhibits no duality gap (for instance under~\cref{asm:CompactFs}), the above condition is also necessary, and if the dual attains its optimal value it is necessary and sufficient that there is $\ggl$ such that    
		\begin{align*}
			\b\T\ggl \geq \nu,\quad  \Qb-\AA^*(\ggl)-\BB_{2}^*(\w,\Wb)\in \COP \, .
		\end{align*}
	\end{thm}
	\begin{proof}
		Condition (\ref{eqn:QuadraticUnderEst}) can be equivalently stated as 
		\begin{align*}
			\phi(\x)\geq& \w\T\x+\Wb\bullet\x\x\T+\nu \quad \forall \x\in\R^n,\\
			\Longleftrightarrow \quad
			0\leq& 
			\inf_{\x\in\R^n}
			\lrbr{
				\phi(\x)- \w\T\x-\Wb\bullet\x\x\T
			}
			-\nu, \\					
			=&\inf_{\Pb\in\SS^{n}} 
			\lrbr{  
				q(\Pb)-\BB_{2}^*(\w,\Wb)\bullet\Pb \colon \AA(\Pb) =\b,\ \Pb\in\CPP
			}- \nu \\
			\Longleftarrow \quad
			0\leq&\sup_{\ggl}
			\lrbr{
				\b\T\ggl 
				\colon 
				\Qb-\AA^*(\ggl)-\BB_{2}^*(\w,\Wb)\in \COP
			}-\nu.
		\end{align*}
		For the first implication we merely used the definitions of an infimum and $\phi$. The equality is a consequence of \cref{asm:CharacterizeG} and \eqref{eqn:chlifting}. The second one-sided implication is a consequence of weak duality, with the reverse implication being valid in case the duality gap vanishes. 
		The theorem then is a consequence of the definition of a supremum, in terms of the limit (as $\ell\to\infty$) of a feasible improving sequence, while under strong duality the supremum is in fact attained. 
	\end{proof}
	
	This theorem tells us that by searching the dual feasible set in \eqref{eqn:DualforQuadraticCuts} we may find all the quadratic underestimators of $\phi$, convex or not, at least in the limit. Thus, if a convex quadratic underestimator for $\phi$ exists at a point $\x$, we may find it by solving the dual under additional positive semidefiniteness constraints on $\Wb$. One challenge we seek to meet in future research is the question of how to do this efficiently. For now, we leave the reader with some illustrative examples.

	\begin{exmp}\label{exmp:QuadraticExample}
		For this example, we are investigating 
		\begin{align*}
			\phi(\x) 
			\coloneqq 
			\inf_{\y}
			\lrbr{
				2\x\T\Bb\y + \y\T\Cb\y \colon \e\T\x+\e\T\y = 1,\ \y\in\R^{n_y}_+
			},
		\end{align*}
		which by \cref{thm:Exact2ndStConvx} can be exactly represented as 
		\begin{align*}
			\hspace{-1cm}
			\phi(\x)
			=& 
			\inf_{\Yb,\Zb,\y}
			\lrbr{
				2\Bb\bullet\Zb+\Cb\bullet\Yb \colon 
				\begin{array}{l}
					\e\T\x+\e\T\y = 1,\\
					\Eb\bullet\x\x\T+2\Eb\bullet\Zb+\Eb\bullet\Yb =1,\\
				\end{array}	
				\begin{bmatrix}
					1 & \x\T& \y\T \\
					\x& \x\x\T & \Zb\T\\
					\y& \Zb & \Yb 
				\end{bmatrix}\in\CPP(\R_+^n)
			}\\
			=&
			\sup_{\ggl,\gl_0,\w,\Wb}
			\lrbr{
				\e\T\ggl-\gl_0 + \w\T\x + \x\T\Wb\x
				\colon 
				\begin{bmatrix}
					\gl_0 & -\tfrac{1}{2}\left(\w+\gl_1\e\right)\T & -\tfrac{1}{2}\gl_2\e\T\\
					-\tfrac{1}{2}\left(\w+\gl_1\e\right) & -\left(\Wb+\gl_2\Eb\right)& \left(\Bb-\gl_2\Eb\right)\T \\
					-\tfrac{1}{2}\gl_2\e & \Bb-\gl_2\Eb & \Cb-\gl_2\Eb
				\end{bmatrix}\in\COP(\R_+^n)
			}.	
		\end{align*}
		Setting $n_x=n_y=2$ we randomly generated a random positive semidefinite matrix $\Cb$ with entries given by $1.5566, 0.5781,  0.2557$ and $\Bb$ was set to the identity matrix times minus one. We used this configuration to create the graphs in \cref{fig:QuadraticUnderestimators}. \begin{figure}[h]
			\hspace{-1.5cm}
			\begin{tabular}{cccc}
				\includegraphics[scale= 0.245]{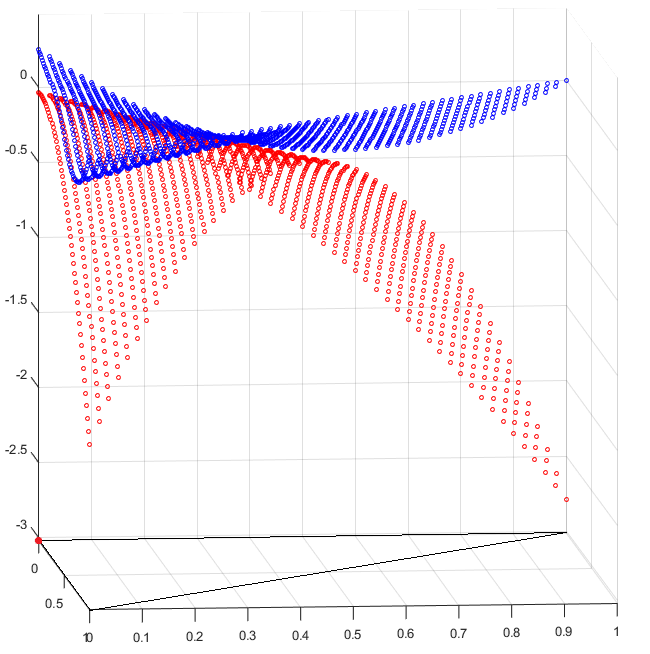} &
				\includegraphics[scale= 0.23]{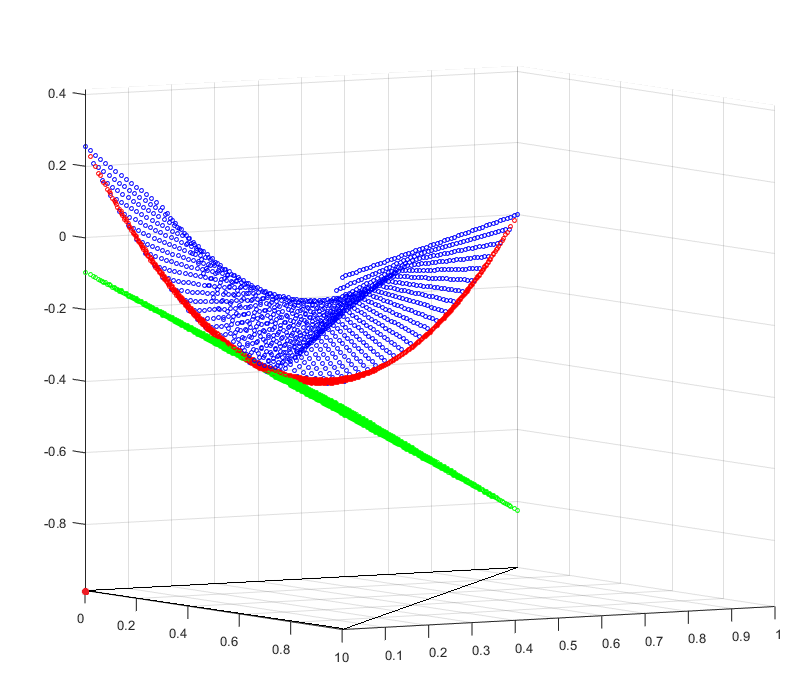} &   
				\includegraphics[scale= 0.23]{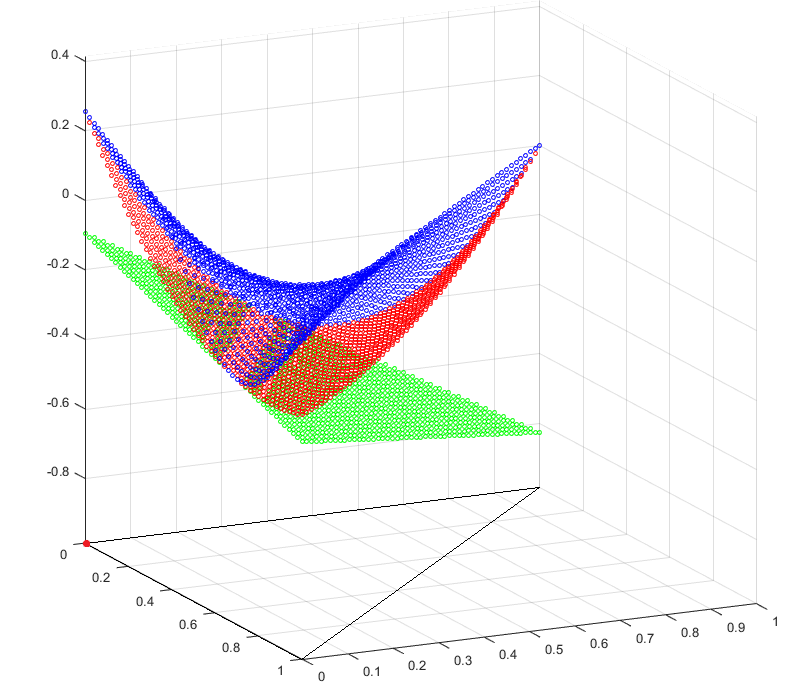} &
				\includegraphics[scale= 0.23]{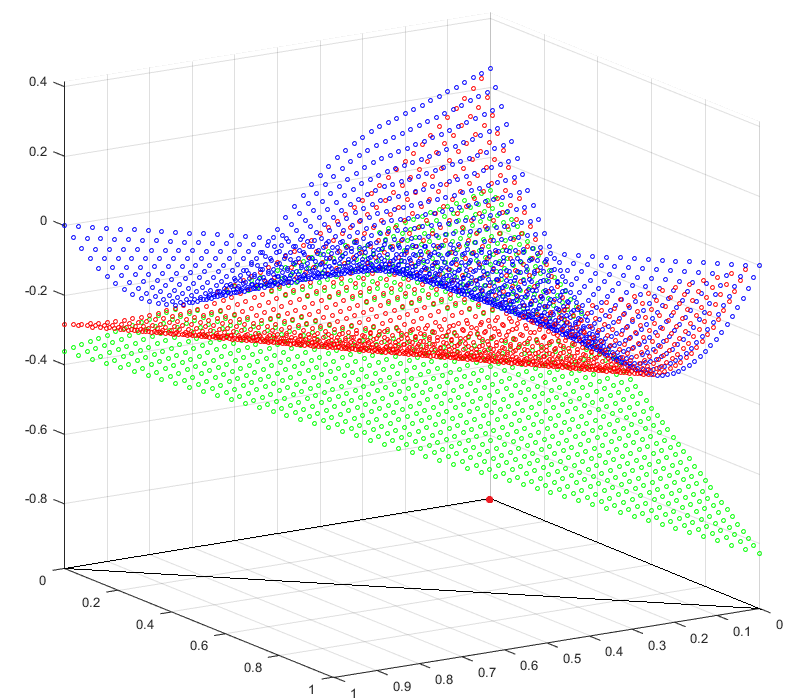}
			\end{tabular}
			\caption{An optimal value function with an affine, convex, and nonconvex quadratic underestimator}
			\label{fig:QuadraticUnderestimators}
		\end{figure}
		
		Firstly, in all the pictures the blue surface represents $\phi$ evaluated across its effective domain, which is $\conv(\lrbr{\oo}\cup\Delta_2)$. This set is indicated in all the pictures as a black triangle, the origin marked by a red dot. We can see that the function $\phi$ is highly nonconvex. In the center, it even seems to be non-smooth, but we did not confirm this analytically. All other surfaces are underestimators constructed at an approximation point $\x_1 = \tfrac{1}{4}[1,1]\T$. 
		
		In the first picture, we solved the exact dual representation of $\phi$, and used the optimal values for the dual variables to parametrize the quadratic underestimator depicted in red. We see that it is nonconvex, and the curvature of the blue surface suggests that there are no convex quadratic underestimators that agree with $\phi$ at $\x_1$. 
		
		For the other three pictures, we first evaluated the closed convex envelope of $\phi$ at $\x_1$ via the dual of $\widehat{\phi}$ as described in \cref{sec:Underestimating via closed convex envelopes}. From the dual solution, we were able to generate the affine underestimator depicted in green. This function supports the closed convex envelope at $\x_1$ but differs from $\phi$ at that point since that function is nonconvex around $\x_1$. 
		
		Finally, we also generated a convex quadratic underestimator depicted in red in the second three pictures in the following way. We solved a version of the exact dual representation of $\phi$, where we fixed the objective to $\widehat{\phi}(\x_1)$, i.e. the value of the closed convex envelope of $\phi$ at $\x_1$, via an additional constraint and we also added the constraint $\Wb\in\SS_+^{n_x}$. The objective was replaced with $\x_2\T\Wb\x_2$ where $\x_2 = \tfrac{1}{2}[1,1]\T$. Thus, we were searching among all quadratic underestimators (i.e. in the feasible set of the dual representation) that are also convex and that support the closed convex envelope at $\x_1$ for the one that curves upwards in the direction of $\x_2$ the most. We see in the pictures that the so-obtained convex quadratic underestimator tracks the convex envelope of $\phi$ much closer than the affine one. 
	\end{exmp}
	
	\begin{exmp}\label{exmp:QuadraticExample2}
		We will now investigate the quadratic underestimators of $\phi$ as described in \cref{exmp:QuadraticExample} but with $n_x=n_y=1$, so that 
		$$
		\phi(x) = 
		\inf_{y\in\R}
		\lrbr{
			2Bx y+Cy^2
			\colon
			\begin{array}{l}
				x+y = 1,\ \\ 
				x\in\R_+,\ y\in\R_+ 
			\end{array}            
		}
		= 
		\begin{cases}
			(C-2B)x^2+2(B-C)x+C    & \mbox{if } x\in[0,1],\\
			\infty & \mbox{\red{otherwise}.}
		\end{cases}
		$$ 
		Also in this case we have $\FF = \lrbr{[x,y]\T\in\R^2_+\colon x+y = 1 }$. 
		
		In this example, we will also consider a naive alternative to the reformulations derived in \cref{sec:Quadratic first stage cuts via exact second stage convexification}. We can interpret $\phi(x)$ as a quadratic optimization problem in $y$ over the feasible set $\FF(x)\coloneqq \lrbr{y\in\R_+ \colon y =1-x}$ so that by \cref{thm:BurerCPP} we can equivalently reformulate it as 
		\begin{align*}
			\phi(x) 
			& = 
			\inf_{Y,y}
			\lrbr{
				CY+2Byx 
				\colon 
				\begin{array}{l}
					y\hspace{0.1cm} = 1-x, \\
					Y = (1-x)^2, \
				\end{array}				
				\begin{bmatrix}
					1  &y\\					
					y  &Y
				\end{bmatrix}\in\CPP(\R_+^2)
			}.
		\end{align*}
		Hence, we get an alternative reformulation to the one based on \cref{thm:Exact2ndStConvx} depicted in \cref{exmp:QuadraticExample}, which in the present case would be given by 
		\begin{align*}
			\phi(x)
			\coloneqq& 
			\inf_{Y,Z,y}
			\lrbr{
				CY+2BZ\colon 
				\begin{array}{l}
					x+y= 1,\\
					x^2+2Z+Y=1,					
				\end{array}	
				\begin{bmatrix}
					1  & x& y \\
					x& x^2& Z\\
					y & Z & y
				\end{bmatrix}\in\CPP(\R_+^3)
			}.	
		\end{align*}
		The duals of the two representations of $\phi$ are given by
		\begin{align*}
			\phi(x)
			=&\sup_{w,W,\gl_0,\ggl}
			\lrbr{ \gl_1+\gl_2-\gl_0+wx+Wx^2  
				\colon 
				\begin{bmatrix}
					\gl_0 & -\tfrac{1}{2}(w+\gl_1) & -\tfrac{1}{2}\gl_1\\
					-\tfrac{1}{2}(w+\gl_1) & -(W+\gl_2) & B-\gl_2\\
					-\tfrac{1}{2}\gl_1 & B-\gl_2 & C-\gl_2
				\end{bmatrix}\in\COP(\R_+^3)
			}, \\		
			=&
			\sup_{\gl_0,\gl_1,\gl_2}
			\lrbr{
				\gl_1(1-x)+\gl_2(1-x)^2-\gl_0
				\colon
				\begin{bmatrix}
					\gl_0 & Bx-\tfrac{1}{2}\gl_1\\
					Bx-\tfrac{1}{2}\gl_1& C-\gl_2
				\end{bmatrix}\in\COP(\R_+^2)
			},
		\end{align*}
		where both representations are equal to $\phi(x)$ since both its primal representations have a compact feasible set and hence a dual Slater point.
		
		Now we investigate solutions of these duals for $B=0,\ C= 1,\ x=\sfrac{1}{2}$. In this case $\phi(x) = (1-x)^2$, so that $\phi(\sfrac{1}{2}) = \sfrac{1}{4}$ and for the two duals of $\phi$ this optimal value is attained at 
		\begin{align*}
			\left(\gl_0,\gl_1,\gl_2,w,W\right) = \left(\tfrac{1}{4},\tfrac{3}{4},\tfrac{1}{8},-\tfrac{1}{2},-\tfrac{1}{2}\right), \
			\left(\gl_0,\gl_1,\gl_2\right) = \left(0,0,1\right),
		\end{align*} 
		where feasibility is easily checked by the respective matrix blocks being positive semidefinite. The quadratic functions in $x$ obtained from the objective functions after plugging in these values are $\tfrac{7}{8}-\tfrac{1}{2}x-\tfrac{1}{2}x^2$, and $(1-x)^2$ respectively. Both of these equal $\phi\left(\sfrac{1}{2}\right) = \tfrac{1}{4}$ at $x = \sfrac{1}{2}$ and both functions underestimate $\phi$ on the unit interval. However, we see that only the second one equates to $(1-x)^2$ which would be the best quadratic underestimator of $\phi$ over its effective domain, in the sense that it is an exact representation, while the other underestimator is not even convex. 
		
		However, this optimal solution of the full model is not unique. For example, setting
		\begin{align*}
			\hspace{-1cm}
			\left(\gl_0,\gl_1,\gl_2,w,W\right) =
			\left(s,2s+1,-(s+\tfrac{1}{4s}),-2,1\right), \mbox{ for any }s>0,
		\end{align*}   
		we can show that the the resulting matrix block $\Mb(s)-\Nb\in\SS_+^3$ for all $s>0$, where $\Nb$ is a nonnegative matrix with zero entries except for $(\Nb)_{23} = \tfrac{1}{2s}$ so that $\Mb(s)\in\SS_+^3+\NN^3 = \COP(\R_+^3)$ and the objective is $2s+1-(s+\sfrac{1}{4s})-s-2x+x^2 = 1+\sfrac{1}{4s}-2x+x^2\rightarrow 1-2x+x^2= (1-x)^2$ as $s\rightarrow \infty$. Indeed positive semidefiniteness of 
		\begin{align*}
			\Mb(s)-\Nb = 
			\begin{bmatrix}
				s  &  -s  &  -s \\
				-s &s+\tfrac{1}{4s}&s-\tfrac{1}{4s}\\
				-s &s-\tfrac{1}{4s}&s+\tfrac{1}{4s}
			\end{bmatrix},
		\end{align*} 
		can be inferred by checking all principal minors. The determinant is zero as the first row times $-2$ gives the sum of the other two, the other principal minors are easily calculated and in fact nonnegative for $s>0$. This illustrates the sufficiency part of the first statement in \cref{thm:COPfeasibleQuadUEs}. Numerical experiments suggest that there is no feasible solution where the first dual attains $(1-x)^2$, but we do not bother proving this here. 
		
		Interestingly, the smaller dual can produce $(1-x)^2$, which would suggest, that the smaller model is preferable. However, this is not always the case. For example, had we chosen $C<0$, then any feasible solution necessarily has $\gl_2<0$ so that only nonconvex quadratic underestimators are obtainable. In contrast, to full model is always able to retrieve all possible quadratic understimators in the limit by \cref{thm:COPfeasibleQuadUEs}.

		If in addition, bilinear terms are present, i.e.\ $B\neq0$, then another deficiency becomes apparent. For example in case $B = C = -1$, so that $\phi(x) = x^2-1$, the dual will only ever produce concave quadratic functions. But more importantly, these will not necessarily be underestimators of $\phi$. For example, if we evaluate the dual at $x=\sfrac{1}{2}$, where $\phi(\sfrac{1}{2})= -\sfrac{3}{4}$, it is easy to see that $\gl_0 = 0, \ \gl_1 = -1, \gl_2 = -1$ is an optimal solution, with optimal value given by $-\sfrac{3}{4}$. But the implied quadratic function is given by $-2+3x-x^2$ which overestimates $x^2-1$ on the interval $[\sfrac{1}{2},1]$. The reason for this pathology is the fact that the reduced duals feasible set depends on $x$, so the dual solutions are not feasible for all $x$.
	\end{exmp}
	
    \blue{
    \section{An application example}
		In this section, we will give an example of how to use the theory outlined in this text for numerical purposes. We will devise a Benders-type scheme to tackle a certain nonconvex QCQP, allowing for parallelization of some of the procedures. The so-obtained lower bounds may demonstrate that our results can have significant computational benefits. However, since this is not the main focus of the text, we will not discuss results on whether the algorithmic scheme converges, nor will we consider real-life instances, but merely artificial ones.  
		
		\subsection{The model}
		We will consider the following optimization problem 
		\begin{align}\label{eqn:TheModel}
			v^*\coloneqq \min_{\x\in\R_+^{Sn},\ \x_i\in\R_+^n}
			\lrbr{
				\x\T\Ab\x +2 \a\T\x + \sum_{i=1}^{S}\x _i\T\Ab_i\x _i
				\colon 	
				\begin{array}{l}					
					\x _i\T\Qb_{ij}\x _i= 1 ,\
					j \in\irg{1}{m},\ i\in\irg{1}{S},\\
					\x\T\x\leq r,\ 
					\x = 
					\begin{bmatrix}
						\x _1\T,&\dots,&\x _S\T
					\end{bmatrix}\T
				\end{array}				
			},			
		\end{align}
		where $\Ab\in\SS_+^{Sn},\ \Ab_i\in\SS^{n}, \ i\in\irg{1}{S},\ \Qb_{ij} \in\SS_+^n,\ j\in\irg{1}{m},\ i\in\irg{1}{S}$ and $\a\in\R^{Sn}$. Such a structure may arise if we have a QCQP where the nonconvex part can be separated and decomposed into $S$ separate blocks of variables and constraints. The model can be reformulated into a two-stage form 
		\begin{align*}
			&\min_{\x\in\R^{Sn},\ \x _i\in\R^n}
			\lrbr{
				\x\T\Ab\x +2 \a\T\x + \sum_{i=1}^{S}\phi_i(\x_i)
				\colon 	
				\begin{array}{l}									
					\x\T\x\leq r,
					\x = 
					\begin{bmatrix}
						\x _1\T,&\dots,&\x _S\T
					\end{bmatrix}\T
				\end{array}				
			},			
			\\
			&\phi_i(\x_i)
			= 
			\begin{cases}
				\x_i\T\Ab_i\x_i & \mbox{if } \x_i\in\FF_i\\
				\infty & \mbox{otherwise}
			\end{cases},
			\quad 
			\FF_i\coloneqq 
			\lrbr{
				\x_i\in\R_+^n
				\colon 
				\x _i\T\Qb_{ij}\x _i= 1,\ \x_i\T\x_i\leq r 
			},\quad 
			\forall i\in\irg{1}{S},
		\end{align*}
		and by \cref{thm:ClosedConvexEnvelope} the closed convex envelopes of $\phi_i,\ i\in\irg{1}{S}$ are given by 
		\begin{align*}
			\hat{\phi}_i(\x_i)
			= 
			\min_{\Xb_i\in\SS^n}
			\lrbr{
				\Ab_i\bullet\Xb_i
				\colon 
				(\x_i,\Xb_i)\coloneqq 
				\begin{bmatrix}
					1 & \x_i\T \\ \x_i & \Xb_i
				\end{bmatrix}\in\GG(\FF_i)
			}, \quad i\in\irg{1}{S},
		\end{align*}
		where the shorthand $(\x_i,\Xb_i)$ for the matrix block will be used throughout this section.   
		We see that we can obtain a lower bound $v_{rel}^*$ on $v^*$ by solving the convex relaxation obtained by replacing $\phi_i$ with $\hat{\phi}_i, \ i \in\irg{1}{S}$ and the remainder of this section will be concerned with calculating this lower bound by using a Bender type strategy.   
		
		Before proceeding we would like to point out a couple of noteworthy things. Firstly, we chose not to have a second stage variable here to keep the exposition simple and to demonstrate that a "natural" two-stage form is not necessary for our underestimators to be applicable. Secondly, in the construction of $\FF_i, \ i \in\irg{1}{S}$, we added a ball constraint which is redundant for the overall problem; however, its presence ensures the compactness of the sets. As a consequence, \cref{asm:CompactFs} is fulfilled and $\hat{\phi}_i, \ i\in\irg{1}{S}$ are indeed exactly the closed convex envelopes, their defining convex optimization problems enjoy zero duality gap and, finally, $\dom(\hat{\phi}) = \cl\conv\dom\phi = \conv\FF_i$ by \cref{thm:ClosedConvexEnvelope}.  We will now proceed with presenting a solution scheme for the convex relaxation.

		\subsection{The method}
		We start out by rewriting $\phi_i$ as follows. First, define $\KK_i\coloneqq \mathrm{cone} (\lrbr{1}\times\FF_i), \ i\in\irg{1}{S},$ which is a pointed cone that has first canonical basis vector $\e_1\in\interior\KK^*$ (its dual cone), by construction due to the compactness of $\FF$.  
		Then an elementary argument shows that for any $i\in\irg{1}{S}$ we have: 
		\begin{align*}
                \hspace{-0.5cm}
			\hat{\phi}_i(\x_i)
			= 
			\min_{\Xb_i\in\SS^n}
			\lrbr{
				\Ab_i\bullet\Xb_i
				\colon 
				(\x_i,\Xb_i)\in
				\CPP(\KK_i)
			}
			= 
			\sup_{\ga\in\R,\w\in\R^n}
			\lrbr{
				-\ga-2\x_i\T\w \colon 
				\begin{bmatrix}
					\ga & \w\T\\ \w & \Ab_i
				\end{bmatrix}\in\COP(\KK_i)	
			}. 		
		\end{align*} 
		Define the set of feasible $(\ga,\w)$ as $\DD_i, i \in\irg{1}{S}$, then 
		\begin{align*}
			v_{rel}^*
			= 
			\min_{\substack{\x\in\R^{Sn},\\ \x _i\in\R^n, \ \phi_i\in\R}}
			\lrbr{
				\x\T\Ab\x +2 \a\T\x + \sum_{i=1}^{S}\phi_i
				\colon 	
				\begin{array}{l}									
					\x\T\x\leq r,\
					\x = 
					\begin{bmatrix}
						\x _1\T,&\dots,&\x _S\T
					\end{bmatrix}\T,\\
					-\ga-2\x_i\T\w \leq \phi_i, \quad \forall (\ga,\w)\in\DD_i, \ i \in\irg{1}{S}
				\end{array}				
			},	
		\end{align*}
		which we will tackle by the Benders method. That is, we define the sets 
		\begin{align*}
			\DD^{f}_i\coloneqq& \lrbr{(\ga,\w)\in\R^{n+1}\colon -\ga-2\w\T\x_i\leq 0, \ \forall \x_i\in\dom(\hat{\phi}_i)}, & i\in\irg{1}{S},\\
			\DD^{o}_i \coloneqq& \lrbr{(\ga,\w)\in\R^{n+1}\colon -\ga-2\w\T\x_i\leq \phi_i(\x_i), \ \forall \x_i \in\R^n},& i \in\irg{1}{S},
		\end{align*}
		which are the sets of feasibility and optimality cuts respectively, that will be reconstructed iteratively via finite approximations $\DD^f_{i,k}\subseteq \DD^f_{i},\ \DD^o_{i,k}\subseteq \DD^o_{i}$ in the $k$-th iteration of the Benders procedure.   
		
		Let $v^*_{rel,k}\leq v^*_{rel}$ be defined via the master problem after the $k$-th iteration:  
		\begin{align*}
			v_{rel,k}^*
			= 
			\min_{\substack{\x\in\R^{Sn},\\ \x _i\in\R^n, \ \varphi_i\in\R}}
			\lrbr{
				\x\T\Ab\x +2 \a\T\x + \sum_{i=1}^{S}\varphi_i
				\colon 	
				\begin{array}{l}									
					\x\T\x\leq r,\
					\x = 
					\begin{bmatrix}
						\x _1\T,&\dots,&\x _S\T
					\end{bmatrix}\T,\\
					-\ga-2\x_i\T\w \leq \varphi_i, \quad \forall (\ga,\w)\in\DD^o_{i,k}, \ i \in\irg{1}{S},\\
					-\ga-2\x_i\T\w \leq 0,\  \quad \forall (\ga,\w)\in\DD^f_{i,k}, \ i \in\irg{1}{S}
				\end{array}				
			},	
		\end{align*}	
		with $\x^k,\x_1^k,\dots,\x_S^k,\varphi_1^k,\dots,\varphi_S^k$ as its optimizers, which exist as long as we can identify a single optimality cut in the first iteration since the set of feasible value for $\x$ is compact. We can solve the dual representations $\phi_{i}(\x_i^k), \ i \in\irg{1}{S}$ to obtain $(\ga^k_i,\w^k_i)$ for augmenting $\DD^o_{i,k+1}\coloneqq \DD^o_{i,k}\cup\lrbr{(\ga^k_i,\w^k_i)}$ whenever the respective problem is bounded. Otherwise, we solve 
		\begin{align}\label{eqn:ImprovingRay}
			\sup_{\ga\in\R,\w\in\R^n}
			\lrbr{
				-\ga-2\x_i\T\w \colon 
				\begin{bmatrix}
					\ga & \w\T\\ \w & \Ob
				\end{bmatrix}\in\COP(\KK_i),\
				\|[\ga,\w\T]\T \|_2\leq 1 
			},
		\end{align}
		in order to identify an improving ray to likewise augment $\DD^f_{i,k+1}$. The fact that solutions to \eqref{eqn:ImprovingRay} allow us to construct feasibility cuts stems from 
		\cite[Proposition 1. and Proposition 2.]{saito_benders_2007}
		which are applicable since by \cref{asm:CompactFs} a Slater point in the dual representation of $\hat{\phi}_i$ exists independently of the value of $\Ab_i$ (in fact $\e_1\e_1\T\in\interior\COP(\KK_1)$ since $\e_1\in\interior\KK_i^*$ so that $\ga$ can always be scaled as to attain a Slater point regardless of the value of $\Ab_i$). In any case, we can shrink the feasible set of $v^*_{rel,k}$ until we find that $\hat{\phi}_i(\x_i^{k}) = \varphi_i^{k}, \ i \in\irg{1}{S}$ in which case no new cuts can be generated and the procedure stops.

		\begin{algorithm}[H]
			\SetAlgoLined
			\KwResult{$v^*_{rel}$ }
			set $k=1$,\;
			construct initial solution $(\x^1)\T \coloneqq [(\x_1^1)\T,\dots,(\x_S^1)\T]\T$,\;
			set $\varphi^1_i,\ i \in\irg{1}{S}$ equal to $-\infty$,\;
			\Repeat{$\hat{\phi}_i(\x^k)=\varphi_i^k, \ i\in\irg{1}{S}$.}{
				solve $\hat{\phi}_i(\x^k), \ i \in\irg{1}{S}$,\;
				\For{$i\in\irg{1}{S} \colon \hat{\phi}_i(\x^k)> \varphi_i^{k}$}
				{
					obtain optimal $(\ga,\w)$ (or the respective improving ray) from the dual representation of $\hat{\phi}_i(\x_i^k)$ and update $\DD^o_{i,k+1} = \DD^o_{i,k}\cup\lrbr{(\ga,\w)}$ (or $\DD^f_{i,k+1}$ respectively),
				}
				set $k = k+1$,\;
				solve $v^*_{rel,k}$ to obtain $\x^k$, 
			}
			\caption{\footnotesize{Solving the convex relaxation via Bender's Cuts }}
		\end{algorithm}

		\begin{rem}
			It is possible to initialize the algorithm with $\x_1$ obtained from $v^*_{rel}$ where, for all $i\in\irg 1S$, the sets $\DD^o_{i,1}=\DD^f_{i,1}=\emptyset$ are empty and $\varphi_i$ are subject to a trivial lower bound (or absent altogether). We found the method to perform better when we calculated initial solutions $\x_i^1\in\FF_i, \ i \in\irg{1}{S}$ via {\tt Gurobi} and then concatenate these vectors to $\x^1$. Solving this feasibility problem is not harder than the QCQP subproblems that we have to solve anyway when solving the copositive dual (see next section), so this effort seemed reasonable to us.  
		\end{rem}

		\subsubsection{Solving the copositive subproblem}
		
		To solve the copositive dual, we will adapt a cutting plane strategy recently studied in \cite{gabl_solving_2023}, for the case where copositivity was meant in the classic sense, i.e., with respect to the nonnegative orthant. We will generalize their approach here and discuss the algorithm for a fixed $i\in\irg{1}{S}$ since the procedure is the same across all instances. 
		We start by constructing a polyhedral outer approximation $\CC_1\subseteq \COP(\KK_i)$ which is updated in every iteration so that in the $k$-th iteration we solve
		\begin{align*}
			\theta(\x_i,\CC_k)\coloneqq\sup_{\ga,\w,\Sb_k}
			\lrbr{
				-\ga-2\x_i\T\w \colon 
				\begin{bmatrix}
					\ga & \w\T\\ \w & \Ab_i
				\end{bmatrix}\eqqcolon \Sb_k\in\CC_k	
			}. 	
		\end{align*}
		Then, a copositivity check either certifies $\Sb_k\in\COP(\KK_i)$ or gives a negative certificate $\z_k\in\KK_i$ for which $\z_k\T\Sb_k\z_k<0$ in which case the reverse constraint $\z_k\T\Sb_k\z_k\geq 0$ can be used as a cutting plane to update $\CC_{k+1}\coloneqq \CC_k\cap\lrbr{\Sb\in\SS^{n+1}\colon \z_k\T\Sb\z_k\geq 0 }$. The procedure is repeated until set-copositivity is successfully certified and can be summarized as follows.  
		
		\begin{algorithm}[H]
			\SetAlgoLined
			\KwResult{$v^*$ }
			set $k=1$\;
			construct outer approximation $\CC_1\supseteq \COP(\KK_i)$, \;
			\Repeat{$\Sb_k\in\COP(\KK_i)$}{
				solve $\theta(\x_i,\CC_k)$ to obtain $\Sb_k$,\;
				check $\Sb_k\in\COP(\KK_i)$,\;
				\If{$\Sb_k\notin\COP(\KK_i)$}{
					obtain certificate $\z_k\in \KK_i$,\;
					set $\CC_{k+1}$ = $\CC_k\cap\lrbr{\Sb\in\SS^{n+1}\colon \z_k\T\Sb\z_k \geq 0}$, 
				}
			}
			\caption{\footnotesize{Solving set-copositive optimization problems}}
		\end{algorithm}
		
		To check set-copositivity it suffices to check whether 
		\begin{align*}
			\begin{bmatrix}
				1 \\ \z
			\end{bmatrix}\T
			\begin{bmatrix}
				\ga & \w\T \\ \w  & \Ab_i
			\end{bmatrix}
			\begin{bmatrix}
				1 \\ \z
			\end{bmatrix} = 
			\ga +2\w\T\z + \z\T\Ab_i\z \geq 0, \quad \forall \z\in\FF_i,
		\end{align*}
		since for $\z_k\in\KK_i$ we have $\z_k\T\Sb\z_k\geq 0$ if and only if $(\e_1\T\z_k)^{-2}\z_k\T\Sb\z_k\geq 0$ and $(\e_1\T\z_k)^{-1}\z_k\in\KK_i$ since $\e_1\T\z_k>0$ whenever $\z_k\in\KK\setminus\lrbr{\oo}$ by construction. To check whether the semi-infinite constraint holds we can  simply solve 
		\begin{align*}
			\min_{\z\in\R_+^n}
			\lrbr{
				\ga +2\w\T\z + \z\T\Ab_i\z 
				\colon 
				\ \z\T\Qb_{ij}\z = 1,\ j\in\irg{1}{m}, \ \z\T\z\leq r
			}.
		\end{align*}
		If it is nonnegative, set-copositivity is certified successfully. Otherwise the optimizer $\z^*$ can be used to update $\CC_k$ using $\z_{k+1}\coloneqq [1,(\z^*)\T]\T$. Also, note that for the sake of producing improving rays that populate the feasibility cuts, we can apply the same procedure to \eqref{eqn:ImprovingRay} in an analogous manner.  
		
		\begin{rem}
			The copositivity test still requires us to solve a nonconvex QCQP, in other words to solve the original QCQP with our proposed method we have to solve many QCQPs in the process. The benefit comes from the fact that the scale of the QCQPs we have to solve is much smaller. In our example, the size of the smaller QCQPs in the subproblems is $n$ while the original QCQP is of size $Sn$, which my be significant if $S$ is large enough. Thus, if the small-scale QCQPs can be managed well, our method allows us to leverage this capability for the sake of attacking the larger problem. To the best of our knowledge, this is the first framework that enables such an approach in this generality. 
		\end{rem}
		
		\begin{rem}
			
   Let $K$ denote the index of the last iteration of the cutting plane procedure. Since  $\z_k$  are feasible for the respective $\FF_i$ for all $k \in\irg{1}{K}$, the procedure above, as a byproduct, produces potentially feasible solutions for the original problem. In our experiments we took the last iterate $\z_K^i$ from each of the $S$ subproblems and concatenated them into a vector $\x_{ub} = [(\z_K^1)\T,\dots,(\z_K^S)\T]\T$, which in case $\x_{ub}\T\x_{ub}\leq r$ was used to calculate an upper bound on the original problem. As we will see later, this ad-hoc approach performed reasonably well. 
		\end{rem}
		
		\begin{rem}
			Finally, we would like to remark that solving the $S$ copositive duals may be parallelized, which potentially saves computation time. We did not implement such a parallelization, but in our experiments, we calculated what the computation time would have been if we merely accounted for the longest of the $S$ dual computations in every iteration, which gives an optimistic estimate of the computation time under parallelization.  
		\end{rem}

		\subsubsection{A stabilizing regularization}
		During our numerical experiments we noticed that even if it was the case that an iterate $\x_i^k$ was feasible for $\conv\FF_i = \dom\hat{\phi}_i$ by any reasonable numerical standard, the copositive dual was often still unbounded so that another feasibility cut would be generated, which eventually led the numerical instabilities and a failure to converge. To tackle this issue in an ad-hoc manner, we slightly modified the definition of $\hat{\phi}_i$, and consequently its dual representation, in the following way, involving a user-defined parameter $\eps>0$: 
		\begin{align*}
			\hat{\phi}_i(\x_i)
			=& 
			\min_{\Xb_i\in\SS^n,\v\in\R^n}
			\lrbr{
				\Ab_i\bullet\Xb_i
				\colon 
				(\x_i+\v,\Xb_i)\in\GG(\FF_i),\
				\|\v\|_2 \leq \eps 
			}\\
			=& 
			\sup_{\ga\in\R,\w\in\R^n}
			\lrbr{
				-\ga-2\x_i\T\w + \eps\|\w\|_2 \colon 
				\begin{bmatrix}
					\ga & \w\T\\ \w & \Ab_i
				\end{bmatrix}\in\COP(\KK_i)	
			}.
		\end{align*}
 
		The advantage of this formulation is that $\hat{\phi}_i(\x_i^k)$ is finite even if $\x_i^k$ is only within an $\eps$-distance of the effective domain of $\hat{\phi}_i$.  Note, that the so-defined version of $\hat{\phi}_i$ is an underestimator of the original one, and that any feasible solution of the newly regularized dual is feasible for the original dual. Thus, by solving this dual we are still able to produce valid optimality cuts. In our experiments, we found that $\eps = 0.05$ was enough to eliminate the undesired behavior described above, while the iterates $\x^k$ were eventually close to feasibility not only for $\conv \FF_i$ but actually for $\FF_i$. As stated before, this was an ad-hoc solution and we hope we can analyze the effect of this regularization more closely in future research.

		\subsection{Numerical Evidence}
		
		In order to test the above scheme we implemented it in {\tt MATLAB 2016R} where we used {\tt Gurobi 9.1} as both an LP and QCQP solver and {\tt YALMIP} \cite{lofberg_yalmip_2004} as a modeling language. All computations were carried out on an {\tt Intel(R) Core(TM) i5-9300H} CPU (2.40GHz, 16 GB RAM).

		We generated instances of \eqref{eqn:TheModel} in the following manner. For constructing each of  $\Ab\in\SS_+^{Sn}, \Ab_i\in\SS^{n},\ i\in\irg{1}{S}$, we drew a random square matrix $\tilde{\Ab}$ of appropriate size and entries in $[-0.5,0.5]$ and set the respective matrices equal to $\tilde{\Ab}\tilde{\Ab}\T/4$. The entry of the vector $\a\in\R^{Sn}$ were also drawn randomly from $[-0.5,0.5]$. For the matrices    
		$\Qb_{ij} \in\SS_+^n,\ j\in\irg{1}{m},\ i\in\irg{1}{S}$ we only used diagonal matrices where the entries were randomly drawn from the interval $[2,3]$. To guarantee that the sets $\FF_i$ were nonempty we generated these matrices one by one, checking whether introducing the respective constraint would render the problem infeasible by consulting {\tt Gurobi}. If that was the case the matrix would be discarded and a new one generated and the process repeated until the desired number of constraints was reached. After that, we calculated a feasible solution $\x_{feas}$ and set $r = 2\x_{feas}\T\x_{feas}$. By following this procedure, we made sure that all the instances we constructed were in fact feasible. In this manner, we generated four instances with $S=30,\ n = 10$, and $m\in\irg{5}{8}$.
			
		\begin{table}[htbp]	\label{results}
			\centering{	
				\begin{tabular}{c||rrr|rrr||rr}
					\multicolumn{1}{l||}{} & \multicolumn{ 3}{c|}{\textbf{{\tt Gurobi}} } & \multicolumn{ 3}{c||}{\textbf{Benders} } & \multicolumn{2}{c}{\textbf{Time}}  \\ 
					\multicolumn{1}{l||}{$m$} & \multicolumn{1}{c}{UB} & \multicolumn{1}{c}{LB} & \multicolumn{1}{c|}{Gap} & \multicolumn{1}{c}{UB} & \multicolumn{1}{c}{LB} & \multicolumn{1}{c||}{Gap} & \multicolumn{1}{c}{Total} & \multicolumn{1}{c}{Parallel } \\ \hline
					5 & 133.79 & -359.74 & 368.88 & 105.72 & 78.88 & 25.39 & 9750.3 & 1487.7\\ 
					6 & 143.46 & -299.87 & 309.03 & 129.96 & 107.41 & 17.35 & 7355.5& 1452.8 \\ 
					7 & 150.72 & -318.94 & 311.61 & 137.31 & 130.65 & 4.85 & 4893.8 & 1221.8\\ 
					8 & - & - & - & 130.61 & 127.37 & 2.48 & (5) 1954.4 & 669.5\\ 
				\end{tabular}
								\caption{Results from the experiments.}
			}
		\end{table}
		
		Using the method described in this section, we produced a lower and an upper bound for every instance, which allowed us to also calculate an optimality gap (relative to the upper bound). The time needed for our algorithm to finish was then used as a time limit for {\tt Gurobi} to also attempt to solve the instance. The results are summarized in Table~1. The most important feature of the results is that the lower bound we obtained from the Benders approach was in all instances far better than the one produced by {\tt Gurobi} on its own in the same time frame (see the columns headed by "LB"). This shows that our approach may be beneficial for global optimization techniques which significantly depend on good quality of lower bounds. Interestingly, also the upper bounds (in the columns headed by "UB") are better than what {\tt Gurobi} generated on its own. The last instance was erroneously deemed infeasible by {\tt Gurobi} after 5 seconds (indicated in the table by the number in brackets). 
		
		In the final two columns, we report the total running time of the Benders method as well as the time it would have taken if solving the $S$ duals in every iteration were fully parallelized, based on our optimistic estimate. We see that our method performs better with increasing $m$ both in terms of running time as well as solution quality. The reasons for this, we conjecture, are twofold: firstly the copositivity checks run faster since {\tt Gurobi} benefits from additional constraints, which strengthen the relaxations that it uses internally. Secondly, the smaller $\FF_i$ across $i \in\irg{1}{S}$, the smaller their convex hull, so that the feasibility cuts generated are potentially much stronger, the further the feasible set is restricted. This is of course beneficial for the convex master problem. 
  }

	\section*{Conclusion and future research}
	In this article, we derived affine, and convex/nonconvex quadratic underestimators of optimal value functions of QCQPs. The results generalize existing results in copositive optimization theory and complement existing results from convex duality theory since we can have more than just affine underestimators, even in the nonconvex case. \blue{We also gave a preliminary example of how to exploit these underestimators in a generalized Benders approach, which was designed somewhat ad-hoc but performed reasonably well. }
	In future research, we will explore techniques on how to efficiently separate useful (i.e.\ convex or manageable nonconvex) quadratic underestimators and investigate their utility in global optimization frameworks. 
	
	\section*{Acknowledgements.}
	The research of M.G. is financially supported by the FWF project ESP 486-N. Both authors are indebted to Ivana Ljubic (ESSEC Paris) for inspiring discussions and valuable suggestions. 
	\blue{We would also like to thank the diligent reviewers who in addition to their careful considerations also suggested adding numerical experiments which eventually became an important part of this study.} 
	
	\bibliography{literature}

\begin{thebibliography}{10}

\bibitem{aliprantis_infinite_2006}
C.~D. Aliprantis and K.~C. Border.
\newblock {\em Infinite {Dimensional} {Analysis}: {A} {Hitchhiker}'s {Guide}}.
\newblock Springer, 3rd edition edition, 2006.

\bibitem{anstreicher_convex_2012}
K.~M. Anstreicher.
\newblock On convex relaxations for quadratically constrained quadratic
  programming.
\newblock {\em Mathematical programming}, 136(2):233--251, 2012.

\bibitem{berman_completely_2003}
A.~Berman and N.~Shaked-Monderer.
\newblock {\em Completely positive matrices}.
\newblock World Scientific, 2003.

\bibitem{bomze_solving_2002}
I.~M. Bomze and E.~De~Klerk.
\newblock Solving standard quadratic optimization problems via linear,
  semidefinite and copositive programming.
\newblock {\em Journal of Global Optimization}, 24(2):163--185, 2002.

\bibitem{bomze_interplay_2021}
I.~M. Bomze and M.~Gabl.
\newblock Interplay of non-convex quadratically constrained problems with
  adjustable robust optimization.
\newblock {\em Mathematical Methods of Operations Research}, 93:115--151, 2021.

\bibitem{bomze_uncertainty_2022}
I.~M. Bomze and M.~Gabl.
\newblock Uncertainty {Preferences} in {Robust} {Linear} {Optimization} with
  {Endogenous} {Uncertainty}.
\newblock {\em SIAM Journal on Optimization}, 32(1):292--318, 2022.

\bibitem{bomze_optimization_2023}
I.~M. Bomze and M.~Gabl.
\newblock Optimization under {Uncertainty} and {Risk}: {Quadratic} and
  {Copositive} {Approaches}.
\newblock {\em European Journal of Operational Research}, 310(2):449--476,
  2023.

\bibitem{burer_copositive_2009}
S.~Burer.
\newblock On the copositive representation of binary and continuous nonconvex
  quadratic programs.
\newblock {\em Mathematical Programming}, 120(2):479--495, 2009.

\bibitem{burer_gentle_2015}
S.~Burer.
\newblock A gentle, geometric introduction to copositive optimization.
\newblock {\em Mathematical Programming}, 151(1):89--116, 2015.

\bibitem{cerulli_convergent_2022}
M.~Cerulli, A.~Oustry, C.~D'Ambrosio, and L.~Liberti.
\newblock Convergent {Algorithms} for a {Class} of {Convex} {Semi}-infinite
  {Programs}.
\newblock {\em SIAM Journal on Optimization}, 32(4):2493--2526, Dec. 2022.

\bibitem{dickinson_copositive_2013}
P.~J. Dickinson.
\newblock {\em The copositive cone, the completely positive cone and their
  generalisations}.
\newblock Citeseer, 2013.

\bibitem{dur_copositive_2010}
M.~Dür.
\newblock Copositive programming – a survey.
\newblock In M.~Diehl, F.~Glineur, E.~Jarlebring, and W.~Michiels, editors,
  {\em Recent advances in optimization and its applications in engineering},
  pages 3--20. Springer, Berlin Heidelberg, 2010.

\bibitem{dur_interior_2008}
M.~Dür and G.~Still.
\newblock Interior points of the completely positive cone.
\newblock {\em The Electronic Journal of Linear Algebra}, 17:48--53, 2008.

\bibitem{eichfelder_set-semidefinite_2013}
G.~Eichfelder and J.~Povh.
\newblock On the set-semidefinite representation of nonconvex quadratic
  programs over arbitrary feasible sets.
\newblock {\em Optimization Letters}, 7(6):1373--1386, 2013.

\bibitem{fan_decision_2024}
X.~Fan and G.~A. Hanasusanto.
\newblock A decision rule approach for two-stage data-driven distributionally
  robust optimization problems with random recourse.
\newblock {\em INFORMS Journal on Computing}, 36(2):526--542, 2024.

\bibitem{concavetent}
M.~Gabl.
\newblock Concave tents: a new tool for constructing concave reformulations of
  a large class of nonconvex optimization problems.
\newblock {\em arXiv preprint arXiv:18451}, 2024.

\bibitem{gabl_solving_2023}
M.~Gabl and K.~Anstreicher.
\newblock Solving {Nonconvex} {Optimization} {Problems} using {Outer}
  {Approximations} of the {Set}-{Copositive} {Cone}.
\newblock {\em Optimization Online}, 2023.

\bibitem{gokalp_robust_2017}
C.~Gokalp, A.~Mittal, and G.~A. Hanasusanto.
\newblock Robust convex quadratically constrained quadratic programming with
  mixed-integer uncertainty.
\newblock {\em arXiv preprint arXiv:1706.01949}, 2017.

\bibitem{kim_geometrical_2020}
S.~Kim, M.~Kojima, and K.-C. Toh.
\newblock A geometrical analysis on convex conic reformulations of quadratic
  and polynomial optimization problems.
\newblock {\em SIAM Journal on Optimization}, 30(2):1251--1273, 2020.

\bibitem{locatelli_convex_2014}
M.~Locatelli and F.~Schoen.
\newblock On convex envelopes for bivariate functions over polytopes.
\newblock {\em Mathematical Programming}, 144(1):65--91, Apr. 2014.

\bibitem{lofberg_yalmip_2004}
J.~Löfberg.
\newblock {YALMIP} : {A} {Toolbox} for {Modeling} and {Optimization} in
  {MATLAB}.
\newblock In {\em In {Proceedings} of the {CACSD} {Conference}}, Taipei,
  Taiwan, 2004.

\bibitem{maxfield_matrix_1962}
J.~E. Maxfield and H.~Minc.
\newblock On the matrix equation {X}'{X} = {A}.
\newblock {\em Proceedings of the Edinburgh Mathematical Society},
  13(2):125--129, 1962.
\newblock Edition: 2009/01/20 Publisher: Cambridge University Press.

\bibitem{oustry_convex_2023}
A.~Oustry and M.~Cerulli.
\newblock Convex semi-infinite programming algorithms with inexact separation
  oracles.
\newblock {\em arXiv preprint arXiv:2307.14181}, 2023.

\bibitem{rahmaniani_benders_2017}
R.~Rahmaniani, T.~G. Crainic, M.~Gendreau, and W.~Rei.
\newblock The {Benders} decomposition algorithm: {A} literature review.
\newblock {\em European Journal of Operational Research}, 259(3):801--817, June
  2017.

\bibitem{rockafellar_convex_2015}
R.~T. Rockafellar.
\newblock {\em Convex analysis}.
\newblock Princeton University Press, 2015.

\bibitem{saito_benders_2007}
H.~Saito and K.~Murota.
\newblock Benders decomposition approach to robust mixed integer programming.
\newblock {\em Pacific Journal of Optimization}, 3(1):99--112, 2007.

\bibitem{shapiro_duality_2001}
A.~Shapiro.
\newblock On {Duality} {Theory} of {Conic} {Linear} {Problems}.
\newblock In M.~A. Goberna and M.~A. L\'opez, editors, {\em Semi-{Infinite}
  {Programming}: {Recent} {Advances}}, pages 135--165. Springer US, Boston, MA,
  2001.

\bibitem{xu_improved_2023}
G.~Xu and G.~A. Hanasusanto.
\newblock Improved decision rule approximations for multi-stage robust
  optimization via copositive programming.
\newblock {\em Operations Research}, 0(0), 2023.

\bibitem{yildirim_alternative_2022}
E.~A. Yıldırım.
\newblock An alternative perspective on copositive and convex relaxations of
  nonconvex quadratic programs.
\newblock {\em Journal of Global Optimization}, 82(1):1--20, Jan. 2022.

\end{thebibliography}
	\bibliographystyle{abbrv}

\end{document}